\documentclass{amsart}
\usepackage{amssymb}
\usepackage{amsmath}
\usepackage{subfigure}
\usepackage{epsfig}
\usepackage[all]{xy}

\newtheorem{theorem}{Theorem}[section] 
\newtheorem{lemma}[theorem]{Lemma}     

\newtheorem{proposition}[theorem]{Proposition}

\theoremstyle{definition}
\newtheorem{example}[theorem]{Example}
\newtheorem{remark}[theorem]{Remark}

\newcommand\oink{\mathcal O}

\newcommand\Da{\mathfrak{D}}

\newcommand\bmattrix[4]{\left(\begin{array}{cc}#1&#2\\#3&#4\end{array}\right)}

\newcommand\sbmattrix[4]{\textnormal{\scriptsize$\left(\begin{array}{cc}#1&#2\\#3&#4\end{array}\right)$\normalsize}}

\newcommand\matrici{\mathbb{M}}

\newcommand\q{\mathbf{q}}

\begin{document}

\title[Branches in even characteristic]{Branches in the Bruhat-Tits tree for local fields of even characteristic}
\author{Luis Arenas-Carmona \& Claudio Bravo}

\maketitle

\begin{abstract}
We extend our previous computations for the relative positions of branches of quaternions to the case of local fields 
of even characteristic. This is a key step to understand the set of maximal orders containing a given suborder, which 
 is useful, for instance, to compute relative spinor images, thus solving the selectivity problem. In our previous 
work, the results where given in terms of the quadratic defect. In the present context, we introduce and 
characterize an analogous concept for Artin-Schreier extensions. It is no longer useful to restrict our attention to
orders generated by pure quaternions, as a separable quadratic extension contains no non-trivial element of null trace.
In this work we state our result for an arbitrary pair of generators, for which we discuss a more general version of 
the Hilbert symbol in this context.
\end{abstract} 
\section{Introduction}\label{s1}

Let $K$ be a local field of characteristic two. It is a known fact that $K$ is isomorphic to the field of Laurent series 
$\mathbb{F}_{2^\tau}((\pi))$ over the finite field $\mathbb{F}_{2^\tau}$ for some positive integer $\tau$. 
We assume  $K=\mathbb{F}_{2^\tau}((\pi))$ in all that follows. In particular, the ring of integers is the power series
ring $\mathcal{O}= \mathbb{F}_{2^\tau}[[\pi]]$, $\pi$ is a uniformizing parameter, and 
$\mathbb{K}=\mathbb{F}_{2^\tau}$ is the residue field of $K$. We let 
$\nu: K\rightarrow \mathbb{Z}\cup \lbrace \infty \rbrace$ be the usual valuation and, for any fractional ideal 
$I\subset K$, we denote by $\nu(I)$ the valuation $\nu(a)$ of any element $a$ satisfying $I=a\mathcal{O}$. 
We also use the absolute value $|a|=c^{\nu(a)}$, where $c<1$ is a fix positive real number.  We let $\matrici_2(K)$ 
denote the ring of 2-by-2 matrices with coefficients in $K$,
and we identify the scalar $\lambda$ with the scalar matrix \scriptsize$\bmattrix \lambda00\lambda$ \normalsize
whenever confusion is unlikely or irrelevant.

The Bruhat-Tits tree $\mathfrak{t}(K)$, also called the BTT in the sequel, is defined as a graph
whose vertices are the closed balls $B$ in $K$, with an edge joining two balls if and only if one 
is a maximal proper sub-ball of the other. To talk with ease of pictures, we define the level of any such vertex 
as the integer $-r \in \mathbb{Z}$ such that $B=B_a^{[r]}$ is the ball of center $a$ and radius $|\pi^r|$. 
We also write $\nu(B)=r$, which is an extension of our previous notation as fractional ideals are balls centered 
at $0$.

There is a natural bijection between the vertex set $V_{\mathfrak{t}(K)}$ of this graph and the set of maximal
orders in $\matrici_2(K)$ given by equation (\ref{balltoorder}) in \S1.1.
Let $\Omega\subseteq\mathbb{M}_2(K)$ be an arbitrary order. The set 
$S_K(\Omega)$ of maximal orders in $\mathbb{M}_2(K)$ containing $\Omega$ can be describe as the vertex 
set of a full subgraph $\mathfrak{s}_K(\Omega)\subseteq\mathfrak{t}(K)$ (c.f. \cite[\S 1]{a-cb1}).
The graph $\mathfrak{s}_K(\Omega)$, which we call the branch of $\Omega$ in all that
follows, is usually a tubular neighborhood of a line $\mathfrak{m}_K(\Omega)$, i.e., a thick line with stem $\mathfrak{m}_K(\Omega)$ as defined in \S\ref{s1.1}, except for some specific orders, namely
(c.f. \cite[Prop. 5.3]{Eichler2} and \cite[Prop. 5.4]{Eichler2}):
\begin{enumerate}
\item The scalar ring $\Omega=\mathcal{O}\sbmattrix1001\cong\mathcal{O}$, denoted just by 
$\mathcal{O}$ in what follows,  is contained in every maximal order. 
\item The nilpotent order $\Omega=\oink[u]$, where $u\in\matrici_2(K)\backslash\{0,1\}$ is nilpotent, 
has a branch $\mathfrak{s}_K(\Omega)$ called an infinite foliage (c.f. \cite{a-cb1}), which  can be described
as the union of a strictly increasing sequence of thick lines with a common leaf.
\end{enumerate}

\begin{remark}
 In the context of Diophantine 
approximation, infinite foliages are called horoballs (c.f. \cite{Paulin}).
\end{remark}

An explicit description of the graph $\mathfrak{s}_K(\Omega)$ simplifies the study of embeddings of the order
$\Omega$ into Eichler orders, or any other intersection of maximal orders, like rank-$3$ orders 
\cite[Lemma 3.2]{omeat}.   It also allows us to compute the spinor image and, therefore, explicitly describing
spinor class fields or representation fields. Both play a central role in solving the selectivity problem
\cite[\S2]{rffco}.  

The purpose of the present work is to extend the results of \cite{a-cb1} to local field with even characteristic, i.e.,
 we intend to characterize the graphs $\mathfrak{s}_K(\Omega)$, when $\Omega=\mathcal{O}[q_1,q_2]$ is the 
order generated by two elements $q_1,q_2\in\matrici_2(K)$. 
As in our previous work, our approach includes the use of 
an embedding of the graph $\mathfrak{t}(K)$, or more precisely a suitable subdivision of $\mathfrak{t}(K)$, into 
the graph $\mathfrak{t}(L)$, for a suitable finite field extension $L/K$ that depends on $q_1$ and 
$q_2$ (c.f. \cite[\S 1]{a-cb1}).
 One of our main results extends, to the present setting, the explicit 
formulas obtained in \cite{osbotbtt} and \cite{a-cb1} to compute the numerical invariants that describe 
the branch $\mathfrak{s}_K(\Omega)$ in terms of the relations satisfied by
the generators $q_1,q_2 \in \matrici_2(K)$. 
  In \cite{a-cb1} we restricted ourselves to the case $\mathrm{tr}(q_1)=
\mathrm{tr}(q_2)=0$, but this is too restrictive in the present setting, as a quadratic extension
 might contain no non-trivial
trace-zero elements. For this reason, we assume throughout that $q_1$ and $q_2$ are not scalars 
and that they satisfy the equations
\begin{equation}\label{133}
\begin{array}{l}
q_1^2+a_1q_1+b_1=q_2^2+a_2q_2+b_2=0,\\ \Lambda(q_1,q_2):=q_1(q_2+a_2)+q_2(q_1+a_1)=\lambda, 
\end{array}
\end{equation}
for some elements $a_1,a_2,b_1,b_2,\lambda \in K$. We call the matrix $\Lambda(q_1,q_2)$ the symmetric 
product by analogy with the corresponding concept in odd charasteristic (c.f. \cite[\S 2]{a-cb1}). By an explicit
 computation, using the quaternion involution
\small
\begin{equation}\label{conj1}
A=\bmattrix abcd\mapsto
\overline{A}=\bmattrix dbca,
\end{equation}
\normalsize
we can see that $\Lambda(q_1,q_2)= q_1\bar{q_2}+q_2\bar{q_1}$ is a scalar matrix. 
In fact 
\begin{equation}\label{conj2}
\Lambda\left[\bmattrix abcd,\bmattrix {a'}{b'}{c'}{d'}\right]=(ad'+bc'+cb'+da')\bmattrix1001.
\end{equation}
Note that $\overline{A}= \det(A) A^{-1}$, when $A \in \mathbb{M}_2(K)$ is invertible. The standard generators $q_1=i$ and $q_2=j$ of an even characteristic quaternion algebra are given by 
$\lambda=a_1=0$ and $a_2=1$ in \eqref{133},  see\cite{vigneras}.

In our previous work, the results are given in terms of the quadratic defect. For this reason, we need both,
to study  the quadratic defect  in even characteristic and to introduce an analog for Artin-Schreier extensions, 
which we call the Artin-Schreier defect in this work. Both play a significant role in the new context. 

\vspace{2mm}

\subsection{Conventions on graphs, orders and  Moebius transformations}\label{s1.1}
In all of this work, a graph $\mathfrak{g}$ is a set of vertices $V_\mathfrak{g}$ endowed with a 
symmetric relation "$-_{\mathfrak{g}}$" that we call the neighborhood relation. 
Two vertices $v$ and $v'$ satisfying $v-_{\mathfrak{g}}v'$ are called neighbors.
 A subgraph of $\mathfrak{g}$ is any graph $\mathfrak{h}$  with a vertex set 
$V_{\mathfrak{h}} \subseteq V_{\mathfrak{g}}$ that satisfies the statement
$$\forall v,v'\in V_{\mathfrak{h}}:  \quad \left[ (v-_{\mathfrak{h}}v')\Rightarrow(v-_{\mathfrak{g}}v')\right] .$$
When the converse holds for every pair $(v,v') \in V_{\mathfrak{h}}\times V_{\mathfrak{h}}$, 
the graph $\mathfrak{h}$ is called a full subgraph of $\mathfrak{g}$. All subgraphs in 
this work are assumed to be full. The intersection of a family of full subgraphs is well defined with these conventions, 
and it is also a full subgraph. The valency of a vertex $v\in V_{\mathfrak{g}}$ is defined as the cardinality of its set 
of neighbors. 
When every vertex in a connected graph $V_{\mathfrak{g}}$ has valency two or one, we call the valency-one
vertices the endpoints of $\mathfrak{g}$. 
A finite walk in $\mathfrak{g}$ is a sequence of vertices $\underline{w}= v_0v_1\dots v_r$ satisfying the following conditions:
\begin{enumerate}
\item $v_i-_{\mathfrak{g}}v_{i+1}$, for $i=0,\dots,r-1$, and
\item $v_i\neq v_{i+2}$, for $i=0,\dots,r-2$.
\end{enumerate}
The latter is called the no-backtracking condition. We often emphasize the vertices $v_0$, or initial vertex,
and $v_r$, the final vertex, by saying a walk from $v_0$ to $v_r$.
A graph $\mathfrak{g}$ is called connected whenever there is a walk from every vertex $v_0\in V_\mathfrak{g}$ 
to every vertex $v_r\in V_\mathfrak{g}$. A cycle is a walk $v_0v_1\dots v_r$ for which $v_r=v_0$. A tree is a 
connected graph with no cycles. Equivalently, a graph $\mathfrak{g}$ is a tree if there exists a unique walk 
from $v_0$ to $v_r$ 
for any pair of vertices $(v_0,v_r)\in V_{\mathfrak{g}}\times V_{\mathfrak{g}}$. 
It is easy to see that walks in a tree have no repeated vertices. All graphs considered in this work are trees. 
The integer $r$ above is called the length of the walk, and written $r=:l(\underline{w})$. We consider the trivial 
sequence $v_0$ as a walk of length $0$. 
We also define infinite walks of two types:
\begin{itemize}
\item A single infinite walk is a sequence of the form $\underline{w}=v_0v_1\dots$, with one vertex for each
 natural number, satisfying (1) and (2) as above. We define $l(\underline{w}):=\infty$.
\item A double infinite walk is a sequence with one vertex for each integer, i.e., $\underline{w}=\dots v_{-1}v_0v_1\dots$,
also assuming (1) and (2). By convention, we write $l(\underline{w}):=2\infty$. \end{itemize}
As usual, the double infinite walks $\dots v_{-1}v_0v_1\dots$ and $\dots v'_{-1}v'_0v'_1\dots$ are considered as equal whenever there is a fixed $m\in\mathbb{Z}$ satisfying $v'_t=v_{t+m}$ for every  $t\in\mathbb{Z}$. 
We define an end of the graph 
as an equivalence class of single infinite walks, where $v_0v_1\dots$ and $v'_0v'_1\dots$ are equivalent
when there is a fixed $m$ satisfying $v'_t=v_{t+m}$ for every sufficiently large positive integer $t$. 
In pictures, we usually visualize an end as a   star ($\star$) at the border of the tree. When 
$\mathfrak{h}$ is a full subgraph of $\mathfrak{g}$, there is a natural identification between the set $\partial(\mathfrak{h})$  of ends 
of $\mathfrak{h}$, on one hand, and, on the other, the set of ends of $\mathfrak{g}$ for which a walk in 
$\mathfrak{h}$ (i.e., a walk that 
contains only vertices in $V_{\mathfrak{h}}$) can be chosen as a representative. We exploit this 
identification by  notational abuses of the type $a\in\partial(\mathfrak{h})$, for an end $a\in\partial(\frak{g})$,
 or many of its  verbal equivalents. It is easy to see that the ends of the BTT $\mathfrak{t}(K)$ are 
naturally in correspondence  with the $K$-points of 
the projective line $\mathbb{P}^1$ (c.f. \cite[\S 4]{omeat}).

For any walk $\underline{w}$ in a tree $\mathfrak{g}$, we define the line $\mathfrak{p}_{\underline{w}}$ as the smallest subtree of $\mathfrak{g}$ containing the vertices in the sequence $\underline{w}$. If $\underline{w}$ is not explicit, we just say a line $\mathfrak{p}$. If $\underline{w}=v_0 v_1\cdots v_r$, we also write $\mathfrak{p}[v_0,v_r]=\mathfrak{p}_{\underline{w}}$. Similarly, we denote by $\mathfrak{p}(a,b)$ the graph whose vertices are precisely the vertices in a double infinite walk joining the ends $a,b \in \mathfrak{g}$. The last one is called a maximal 
path, or sometimes simply a path, in the sequel. A ray $\mathfrak{p}[v,a)$ is defined analogously.
The length of a line is the length of the associated walk, and written analogously, e. g., 
$l(\mathfrak{p}_{\underline{w}}) =l(\underline{w})$.
 We say that $v_0$ is an $r$-neighbor of $v_r$ is there exist a 
line of length $r$ whose endpoints are precisely $v_0$ and $v_r$. A tubular neighborhood $\mathfrak{p}^{[n]}$, of some
 line $\mathfrak{p}$, is the subtree of $\mathfrak{g}$ containing precisely the $s$-neighbors of vertices in 
$\mathfrak{p}$, for all $s \leq n$.

As a vertex in the BTT is a ball, which is fully determined by its radius and one center, it is immediate 
that $V_{\mathfrak{t}(K)}$ can be identified with a subset of $V_{\mathfrak{t}(L)}$, for any finite field extension $L/K$.
In what follows, a vertex in $V_{\mathfrak{t}(L)}$ is said to be defined over $K$, if it corresponds to a
vertex in $V_{\mathfrak{t}(K)}$. This definition extends naturally to intermediate fields.
 However, this identification cannot define a morphism of graphs unless $L/K$ is an unramified extension. 
We fix this problem by a normalization of the distance function. For any pair of vertices $(v,v')\in
V_{\mathfrak{t}(L)}\times V_{\mathfrak{t}(L)}$, we define their distance by $\delta(v,v')=
\frac1{e(L/K)}l(\mathfrak{p}[v,v']_L)$, where $\mathfrak{p}[v,v']_L$ is the corresponding
line in $\mathfrak{t}(L)$, and $e(L/K)$ denotes the ramification index. This distance is independent of the field $L$. 
We also normalize the valuation on $L$ in a similar fashion. For example, a uniformizing parameter $\pi_L$ of $L$ 
has the valuation $\nu(\pi_L)=\frac 1{e(L/K)}$. A similar convention applies to the absolute value. A real number is 
said to be defined over an intermediate field $F$ if it is a multiple of the valuation of the
corresponding uniformizer $\nu(\pi_F)$. We apply this definition to both, valuations of elements or
distance between vertices or subgraphs. Next result is trivial but useful:

\begin{lemma}\label{lem11}
Let  $v$ and $v'$ be two vertices of $\mathfrak{t}(L)$ that are both defined over an intermediate field $F$. 
Then, a vertex $v''$  in the line $\mathfrak{p}[v,v']_L$
is defined over $F$ if and only if its distance to either endpoint is defined over $F$.  
\end{lemma}

A line with endpoints defined over $F$ is said to be defined over $F$. The same applies to maximal paths and rays,
replacing endpoints by ends if needed.
Let $\mathfrak{p}(a,b)_L$ be the unique maximal path in $\mathfrak{t}(L)$ whose ends are 
$a,b \in \mathbb{P}^1(L)$. When the extension $L/K$ is Galois, the group
$\mathrm{Gal}(L/K)$ acts on either set $\mathbb{P}^1(L)$ or $\mathfrak{t}(L)$.
 These two actions are compatible, i.e., paths in $\mathfrak{t}(L)$ satisfy the relation
$$\sigma\Big(\mathfrak{p}(a,b)_L\Big)=\mathfrak{p}\Big(\sigma(a),\sigma(b)\Big)_L,\qquad
\forall\sigma\in\mathrm{Gal}(L/K),\ \forall a,b\in\mathbb{P}^1(L)$$ (c.f. \cite[\S 3]{a-cb1}). There is also a natural action of the group $\mathcal{M}(L)$ of Moebius transformations on the set of balls.
For every element $\mu\in\mathcal{M}(L)$ and every ball $B\subseteq L$, we define:
\begin{enumerate}
\item $\mu*B=\mu(B)$, if the latter is a ball, while
\item $\mu*B$ is the smallest ball properly containing $\mu(B^{\mathit{c}})$,
if $\mu(B^{\mathit{c}})$ is a ball.
\end{enumerate}
The complement above is defined in $\mathbb{P}^1(L)$, so $\infty\in B^{\mathit{c}}$ for any ball $B$.
 In fact, if we remove the ball $B$, as a vertex, from the BBT, the remaining graph has $\kappa+1$ connected components, where $\kappa$ is the cardinality of the residue field of $L$, and their corresponding sets of ends are the sets in the decomposition $\mathbb{P}^1(L)=B^{\mathit{c}}\cup B_1\cup\cdots\cup B_r$, where $B_1,\dots,B_r$ are all maximal proper sub-balls of $B$. The action by Moebius transformation permute balls by permuting such decompositions. See 
 \cite{omeat} or \cite{a-cb1} for details. We also have a compatibility property $\sigma(\mu*B)=\sigma(\mu)*\sigma(B)$, for any element $\sigma\in \mathrm{Gal}(L/K)$, any transformation $\mu\in\mathcal{M}(L)$ and any ball $B\subseteq L$.

Following \cite{a-cb1}, vertices in $\mathfrak{t}(L)$ that are not defined over $K$ are called
here ghost vertices. Any maximal path $\mathfrak{p}(a,b)_L\subseteq \mathfrak{t}(L)$ defined over $K$
is identified with the corresponding path in $\mathfrak{t}(K)$. 
Maximal paths that are not of this form are called ghost paths.

To every ball $B=B_a^{[\nu(\rho)]}$ with $a,\rho\in L$, we associate the maximal order 
\begin{equation}\label{balltoorder}
\mathfrak{D}_B=\mathrm{End}_{\oink_L}\left[\left\langle\left(\begin{array}ca\\1\end{array}\right),
\left(\begin{array}c\rho\\0\end{array}\right)\right\rangle \right].
\end{equation}

This defines a one-to-one correspondence between balls and maximal orders that is assumed in all that follows. It is compatible with field extensions when the order $\Da\subseteq\matrici_2(K)$ is identified with $\Da\otimes_\mathcal{O}\mathcal{O}_L\subseteq\matrici_2(L)$. The branch $\mathfrak{s}_K(\Omega)$ is defined as the largest full
subgraph whose vertices are balls corresponding to maximal orders containing $\Omega$.
As usual, we write $\mathfrak{s}_K(q_1,\dots,q_n)$ instead of $\mathfrak{s}_K(\Omega)$
if $\Omega=\mathcal{O}[q_1,\dots,q_n]$ is the order generated by $q_1, \cdots, q_n$. Clearly $\mathfrak{s}_K(q_1,\dots,q_n)= \bigcap_{i=1}^n\mathfrak{s}_K(q_i)$. With these conventions, the vertex set $V_{\mathfrak{s}_K(\Omega)}$ can be identified with
$V_{\mathfrak{s}_L(\Omega')}\cap V_{\mathfrak{t}(K)}$, where $\Omega'= \oink_L \otimes_{\oink} \Omega$. The same applies to other notations like $\mathfrak{s}_K(q_1,\dots,q_n)$ and $\mathfrak{s}_L(q_1,\dots,q_n)$.

A thick line $\mathfrak{p}$ is defined over $K$ if its stem and its depth are defined over $K$. Similarly, an infinite
foliage is defined over $K$ if its end and, at least, one leaf (valency-one vertex) is defined over $K$.
The following results are easy consequence of Lemma \ref{lem11}:

\begin{lemma}\label{lem12}
A thick line is defined over $K$ if and only if the stem and, at least, one leaf are defined over $K$. 
\end{lemma}

\begin{lemma}\label{lem13}
If two thick lines or infinite foliages, or one of each, are defined over $K$, then the (unique) smallest path from
one to the other is defined over $K$.
\end{lemma}

\section{Main results}\label{s2}

For any element $a \in K$, we denote by $p_a(X) = X^2+X+a$ the corresponding Artin-Schreier polynomial. 
Note that 
\begin{equation}\label{asprop}
p_{a+b}(x+y)=p_a(x)+p_b(y)\textnormal{ and }p_a(0)=a.
\end{equation}
 We define the Artin-Schreier defect 
$\mathbb{D}(a)$ of $a$ as the fractional ideal $\mathbb{D}(a)= \bigcap_{h \in K} \big(p_a(h)\big)$. 
We define the quadratic defect $\delta(a)$, as usual, by the formula $\delta(a) = 
\bigcap_{h \in K} (h^2+a)$ (c.f. \cite[\S 63:A]{Om}). 
The following result is an analog for the Artin-Schreier defect of the classical characterization of 
the quadratic defect in terms of quadratic extensions (c.f. \cite[\S 63:A \& Theo. 63.4]{Om}).

\begin{theorem}\label{t21}
The image of the Artin-Schreier defect is the set:
$$ S= \lbrace (0) , \mathcal{O} \rbrace \cup \lbrace (\pi^{-2t+1}): t >0  \rbrace.$$
Furthermore, if an element $b$ in the algebraic closure $\bar{K}$ satisfies $p_0(b)=a$, then the following statements hold:
\item[i.-] $\mathbb{D}(a)=\{0\}$ if and only if $b\in K$.
\item[ii.-] $\mathbb{D}(a)=\mathcal{O}$ if and only if $b$ generates an unramified quadratic extension of $K$. 
\item[iii.-]$\mathbb{D}(a)=(\pi^{-2t+1})$, for some $t>0$,  if and only if $b$ generates a ramified quadratic extension of $K$. 
\end{theorem}

For the sake of uniformity, we set $\mathfrak{m}_K(\Omega)=\mathfrak{s}_K(\Omega)$ if $\mathfrak{s}_K(\Omega)$ 
is an infinite foliage. This is not perfect, but seems to be the alternative that makes statements simpler in 
Theorem 2.2 below.
By the stem length of a thick line we mean the length of its stem. By convention, we set the stem length of
an infinite foliage as $\infty$. As before, this choice is not perfect. 
For a matrix $q$ that is integral over $\oink$ we write $l(q)$ for the stem length of its branch, which is always in 
the set $\{0,1,\infty,2\infty\}$ \cite{Eichler2}.

We say that $f(X)\in K[X]$ is ramified (respectively, unramified) if its decomposition field is ramified (respectively, unramified) over $K$. Similarly, we say that $f(X)\in K[X]$ is separable if its roots in $\bar{K}$ are different. In any other case, we say that $f(X) \in K[X]$ is inseparable. For example, the polynomial $f(X)=X^2+cX+d$ is separable if and only if $c\neq 0$. We define $A^s$ as the set of either reducible or unramified separable polynomials and $B^s$ as the set of ramified separable irreducible polynomials. The sets $A^i$ and $B^i$ are defined analogously in the inseparable case. Note, however, that there is no irreducible inseparable unramified polynomial.
Let $\lfloor a\rfloor$ denote the largest integer not exceeding $a$.

\begin{theorem}\label{t22}
Let $m_i(X) = X^2+a_iX+b_i \in \mathcal{O}[X]$, for $i\in\{1,2\}$, be quadratic polynomials, and let $q_1,q_2\in \mathbb{M}_2(K) \backslash K$ be matrices satisfying $m_1(q_1)=0$, $m_2(q_2)=0$ 
and $q_1\overline{q_2}+q_2\overline{q_1}=\lambda$. 
Consider $\Delta=\Delta(\lambda, m_1, m_2)=
\lambda^2+ a_1 a_2 \lambda+ a_1^2 b_2+ a_2^2 b_1$, and let $d_f$ be the fake distance defined by 
cases as in Table \ref{table 1}, where we use the following convention:
$$t_i = \textnormal{\Huge$\lfloor$} \frac{\nu(I_i)-1}{2}\textnormal{\Huge$\rfloor$}, \textnormal{ where } 
I_i=\left\{\begin{array}{cl}
 \mathbb{D}\left( \frac{b_i}{a_i^2}\right)&\textnormal{ if }m_i\textnormal{ is separable}\\
\delta(b_i) &\textnormal{ otherwise }\end{array}\right.. $$
Then, if $d_f>0$, it equals the distance between the stems, otherwise the length of their intersection is
$\mathrm{min}\{-2d_f,l(q_1),l(q_2)\}$, except in the following cases:
\begin{itemize}
\item When  $m_1$ and $m_2$ are both reducible inseparable polynomials, and $\lambda \neq 0$, 
the depth of $\mathfrak{m}_K(q_1) \cap \mathfrak{m}_K(q_2)$ is $\lfloor \frac{\nu(\lambda)}{2} \rfloor$. Here
the stem of $\mathfrak{m}_K(q_1) \cap \mathfrak{m}_K(q_2)$ is a vertex if $\nu(\lambda)$ is even and
an edge otherwise. If $\lambda=0$, then either
$\mathfrak{m}_K(q_1) \subseteq \mathfrak{m}_K(q_2)$ or $\mathfrak{m}_K(q_2)\subseteq   \mathfrak{m}_K(q_1)$.
\item When  $m_1$ and $m_2$ are both reducible separable polynomials and $d_f=-\infty$, the formula predicts
an intersection of length $2\infty$, a maximal path. However, if $q_1q_2\neq q_2q_1$ we get a ray instead.
\end{itemize}
\end{theorem}

\begin{table}
\[
\begin{tabular}{|c|c|c|}
\hline $Y_1$ & $Y_2$ & $d_f$\\
\hline\hline
$A^s$& $A^s$ & $-\frac{1}{2}\nu\left( \frac{\Delta}{a_1^2a_2^2}\right) $\\ \hline
$A^s$& $A^i$ & $-\frac{1}{2}\nu\left( \frac{\Delta}{a_1^2}\right)$\\ \hline
$A^i$& $A^s$ & $-\frac{1}{2}\nu\left( \frac{\Delta}{a_2^2}\right)$\\ \hline
$A^i$& $A^i$ &$-\frac{1}{2}\nu\left( \Delta\right)$\\ \hline
$A^s$& $B^s$ & $-\frac{1}{2}\nu\left( \frac{\Delta}{a_1^2a_2^2}\right) -t_2$\\ \hline
$A^s$& $B^i$ & $-\frac{1}{2}\nu\left( \frac{\Delta}{a_1^2}\right)+t_2$\\ \hline
$A^i$& $B^s$ & $-\frac{1}{2}\nu\left( \frac{\Delta}{a_2^2}\right)-t_2$\\ \hline
$A^i$& $B^i$ &$-\frac{1}{2}\nu\left( \Delta\right)+t_2$\\ \hline
\end{tabular}
\qquad
\begin{tabular}{|c|c|c|}
\hline $Y_1$ & $Y_2$ & $d_f$\\
\hline\hline
$B^s$& $A^s$ & $-\frac{1}{2}\nu\left( \frac{\Delta}{a_1^2a_2^2}\right) -t_1$\\ \hline
$B^s$& $A^i$ & $-\frac{1}{2}\nu\left( \frac{\Delta}{a_1^2}\right)-t_1$\\ \hline
$B^i$ & $A^s$ & $-\frac{1}{2}\nu\left( \frac{\Delta}{a_2^2}\right)+t_1$\\ \hline
$B^i$& $A^i$ &$-\frac{1}{2}\nu\left( \Delta\right)+t_1$\\ \hline
$B^s$& $B^s$ & $-\frac{1}{2}\nu\left( \frac{\Delta}{a_1^2a_2^2}\right)-t_1-t_2 $\\ \hline
$B^s$& $B^i$ & $-\frac{1}{2}\nu\left( \frac{\Delta}{a_1^2}\right)-t_1+t_2 $\\ \hline
$B^i$& $B^s$ & $-\frac{1}{2}\nu\left( \frac{\Delta}{a_2^2}\right)+t_1-t_2 $\\ \hline
$B^i$& $B^i$ &$-\frac{1}{2}\nu\left( \Delta\right)+t_1+t_2$\\ \hline
\end{tabular}
\]
\caption{The value of $d_f$ for $m_1\in Y_1$ and $m_2\in Y_2$.}\label{table 1}
\end{table}

A natural question that arises at this point is whether there exist matrices satisfying the hypotheses of
the preceding theorem. This can be answered as follows:

\begin{theorem}\label{t24}
Let $\lambda\in K$, $m_1$, $m_2$ and $\Delta$ be as in the Theorem \ref{t22}. We define,
for $x,y,z,w\in K$, with $y,w\neq0$, the expression
\small
\begin{equation}\label{eq2}
C(x,y,z,w) = \frac{1}{yw} \left[   y^2 m_1\left( \frac{x}{y}\right) + w^2m_2\left( \frac{z}{w}\right) + a_1 z y + a_2 x w \right].
\end{equation}
\normalsize
Then, there exist  linearly independent elements $q_1,q_2 \in \mathbb{M}_2(K)\backslash K$ satisfying
the identities in Equation \eqref{133} if and only if any of the following conditions holds:
\begin{enumerate}
\item[i.-] $\Delta\neq 0$, and at least one polynomial, $m_1$ or $m_2$, has a zero in $K$.
\item[ii.-] $\Delta\neq 0$, and there exist two pairs $(x,y)$ and $(z,w)$ in $K \times K^{*}$ satisfying $C(x,y,z,w)=\lambda$.
\item[iii.-] $\Delta=0,$ $a_1 \neq 0$ and $m_1$ has a zero in $K$.
\item[iv.-] $\Delta=0,$ $a_2 \neq 0$ and $m_2$ has a zero in $K$.
\item[v.-] $\Delta=0$, $a_1=a_2=0$.
\end{enumerate}
In the last case, however, the matrices $q_1$ and $q_2$ are contained in a two dimensional subalgebra.
\end{theorem}
If $K$ were not a perfect field, we would need an extra condition $$\Big[K(\sqrt{b_1},\sqrt{ b_2}):K\Big]\leq2$$ in
item (v) above (see \S8).
To prove this theorem, we need to study the semisimplicity of certain $K$-algebras. This is essentially
a generalization of \cite[Ch. II, Th. 1.3]{vigneras} that plays the role of the Hilbert symbol in even characteristic.

\section{Artin-Schreier and quadratic defects}\label{s3}

Let $a \in K$, and let $p_a(X)=p_0(X)+a$ be the Artin-Schreier polynomial (c.f. \S\ref{s2}). 
We write $P_0=p_0(K)$ for the image of $p_0$,  which is an additive subgroup of $K$. Note that 
$b\mapsto p_a(b)$ is a continuous function on the locally compact space 
$K$ satisfying $\lim_{b\rightarrow\infty} p_a(b)=\infty$. Next result is straightforward from (\ref{asprop}) and 
the observation that $\mathbb{D}(a)\subseteq\Big(p_a(h)\Big)$ for any $h\in K$:
\begin{lemma}\label{lema1}
We have $\mathbb{D}(a+b)\subseteq \mathbb{D}(a)+\mathbb{D}(b)$, for any pair of elements $a,b\in K$. In particular, $\mathbb{D}(a) = \mathbb{D}(a+c)$ for any $c\in P_0$.
Also, for any element $a\in K$, there exists $h=h(a)$ satisfying $\mathbb{D}(a)=\big(p_a(h)\big)$.
\end{lemma}

\begin{example}\label{e33}
$p_{\pi}(X)$ has a zero in $K$, since the image $\overline{p}_{\pi}(X) = X(X+1) \in \mathbb{K}[X]$ has two different roots in the residue field $\mathbb{K}$, and Hensel´s Lemma applies. We conclude that $\mathbb{D}(\pi)=(0)$.
\end{example}

\subparagraph{Proof of Theorem 2.1}
Let $s \in \mathbb{Z} \cup \lbrace \infty \rbrace$, and set $\pi^{\infty} =0$. By Lemma \ref{lema1}, if $\mathbb{D}(a) = (\pi^s)$, there exist $b \in K$ satisfying $p_a(b)= u\pi^s$, for some $u \in \mathcal{O}^{*}$. In particular $p_{a+u\pi^s}(b)=0$, i.e., $\mathbb{D}(a+u\pi^s)=\lbrace 0 \rbrace$. If $s>0$, or if $s= \infty$, we conclude, reasoning as in Example \ref{e33}, that the polynomial $p_{u\pi^s}(X) = p_0(X)+u\pi^s$ has a root in $K$, in particular $\mathbb{D}(u\pi^s)=\lbrace 0 \rbrace$. It follows from Lemma \ref{lema1} that $\mathbb{D}(a) = \lbrace 0 \rbrace$. Now, assume that $s=-2t$, for $t \in \mathbb{Z}_{>0}$, and write $u=b_0+ \epsilon $, where 
$b_0 \in \mathbb{F}_{2^\tau}^{*}$ and $|\epsilon|< 1$. As 
$\mathbb{F}_{2^\tau}$ is a perfect field, we have $b_0= a_0^2$, 
for some $a_0 \in \mathbb{F}_{2^\tau}$. This implies that 
$$p_a(b+a_0\pi^{-t})=p_a(b)+p_0(a_0\pi^{-t})=\pi^{-2t}\epsilon+a_0\pi^{-t}$$ has a larger valuation than $p_a(b)$. This contradicts the fact that $\mathbb{D}(a)=\Big(p_a(b)\Big)$. Now, we assume that $s=-2t+1$, where $t>0$. Let $L$ be the decomposition field of 
$$p_{a}(X)=p_0(X+b)+p_a(b)=p_0(X+b)+u\pi^{-2t+1},$$
and let $\alpha \in L$ be one of its roots.  
Then $\pi^t(\alpha+b)$ satisfies the Einsenstein polynomial $q(X) = X^2+\pi^t X+ u\pi$. This implies that $L$ 
ramifies over $K$. On the other hand, for an arbitrary $t>0$, we let $a=\pi^{-2t+1}$. It is clear that 
$\mathbb{D}(a) \subseteq (\pi^{-2t+1})$. Suppose that equality fails to hold. In particular, 
$|p_a(b)|<|\pi^{-2t+1}|$, 
and by Dominance Principle we have $|b(b+1)|= |p_a(b)+a| = |\pi^{-2t+1}|>1$. This implies that 
$|b|=|b+1|>1$, and then $|b|^2=|p_0(b)| = |\pi^{-2t+1}|$, which is absurd. We conclude that $\mathbb{D}(a)=(\pi^{-2t+1})$. This proves that $(\pi^{-2t+1})$ is in the image of the Artin-Shreier defect for any 
positive integer $t$.
Finally, we assume that $\mathbb{D}(a)= \mathcal{O}$ and  choose $L$ and $\alpha$ as before. 
Then $\alpha+b$ satisfy the polynomial $p_u(X)=X^2+X+u$. Note that $\overline{p_u}(X)$ must be irreducible in $\mathbb{K}[X]$ by Hensel's lemma. We conclude that $L$ is an unramified extension of $K$.
\qed

Note that $\mathbb{D}(a) = \mathcal{O}$ is the smallest possible non trivial Artin-Schreier defect, 
and it is attained when the roots of $p_a(X)$ generate an unramified quadratic extension. 
These elements $a$ play an analog role to that of units of minimal quadratic defect for a local field $K$ 
with odd characteristic.

Let $\delta(a) = \bigcap_{h \in K} (h^2+a)$ be the quadratic defect of $a$ (c.f. \S2).
We write $a = \sum_{i=-2N}^{\infty} a_i \pi^{i}$, where $a_i \in \mathbb{F}_{2^\tau}$. Set $a_i=b_i^2$, wich 
can always be done, since $\mathbb{F}_{2^\tau}$ is perfect. If $\xi= \sum_{i=-N}^{\infty} b_{2i}\pi^i$, 
then $a+\xi^2$ has a nonzero coefficient only for odd exponents $\pi^i$. In particular, either $a$ is a square of $K$, 
or otherwise $|a+\xi^2|=|\pi|^{2t+1}$, for some $t \in \mathbb{Z}$. Note that this absolute value is optimal 
by the Dominance Principle, since squares cannot have an odd valuation. Next result follows:

\begin{proposition}\label{prop 3.4}
The image of the quadratic defect $\delta$ is exactly the set $\lbrace (\pi^{2t+1}) : t \in \mathbb{Z} \rbrace$, and we have $|\delta(a)|\leq |a|$, for any element $a \in K$. Furthermore, there exists $b\in K$ satisfying $|\delta(a)|=|\delta(a+b^2)|=|a+b^2|$.
\end{proposition}

\section{Branches and stems in characteristic 2}\label{s4}

Let $q \in \mathbb{M}_2(K)$ be a matrix that is integral over $\mathcal{O}$, i.e., its irreducible polynomial has 
integral coefficients. We say that $q$ is separable if $m_q(X)=\text{irr}_{q,K}(X)=X^2+cX+d \in \mathcal{O}[X]$ 
is separable, i.e: $c \neq 0$.  Otherwise, we say that $q$ is inseparable.
Suppose that $m_q(X)$ is separable and splits over $K$. Let $\alpha$ be a root of $m_q(X)$. Note that $\alpha+c$ is the other root. Note also that $w=\frac{q+\alpha}{c}\in\matrici_2(K)$ is an idempotent. Then $\mathfrak{s}_K(q)=\mathfrak{s}_K(w)^{[\nu(c)]}$, and $\mathfrak{s}_K(w)$ is a path in $\mathfrak{t}(K)$ (c.f. \cite[\S 4]{a-cb1}). This implies that $\mathfrak{s}_K(w)$ is the stem of $\mathfrak{s}_K(q)$.

\begin{figure}
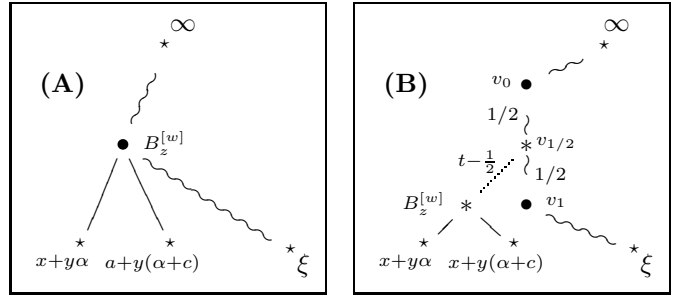

\[ 
\fbox{ \xygraph{
!{<0cm,0cm>;<.8cm,0cm>:<0cm,.8cm>::} 
 !{(0.5,0) }*+{{}^{x+y\alpha}}="e2" !{(2.5,4) }*+{\infty}="e3"  !{(0.8,0.3) }*+{{}^{\star}}="f2" !{(2.2,3.6) }*+{{}^{\star}}="f3" 
!{(2,0) }*+{{}^{a+y(\alpha+c)}}="e4" !{(4.5,0) }*+{\xi}="e5" !{(2.3,0.3) }*+{{}^{\star}}="f4" !{(4.3,0.2) }*+{{}^{\star}}="f5"
  !{(1.5,2) }*+{\bullet}="t3"   !{(1.5,2) }*+{\bullet}="t4" 
!{(2.2,2) }*+{{}^{B_z^{[w]}}}="ti3"
!{(0.5,3) }*+{\textbf{(A)}}="ti99"
   "f2"-"t4" "t3"-@{.}"t4" "f3"-@{~}"t3"
"t4"-"f4" "t3"-@{~}"f5"
 } }
\quad
\fbox{ \xygraph{
!{<0cm,0cm>;<.8cm,0cm>:<0cm,.8cm>::} 
 !{(0.5,0) }*+{{}^{x+y\alpha}}="e2" !{(4,4) }*+{\infty}="e3"
!{(2,0) }*+{{}^{x+y(\alpha+c)}}="e4" !{(4.5,0) }*+{\xi}="e5"
  !{(2.5,2) }*+{*}="t3"    !{(2.5,3) }*+{\bullet}="t6"    !{(2.5,1) }*+{\bullet}="t7" 
  !{(1.5,1) }*+{*}="t4"  !{(1.7,1.7) }*+{{}^{t-\frac{1}{2}}}="t5" 
!{(3,2) }*+{{}^{v_{1/2}}}="ti3" !{(2.1,3.0) }*+{{}^{v_0}}="ti3" !{(3,1.) }*+{{}^{v_1}}="ti3"
!{(0.8,1) }*+{{}^{B_z^{[w]}}}="ti4"
!{(0.5,3) }*+{\textbf{(B)}}="ti99"
!{(3.8,3.6) }*+{{}^{\star}}="f3" !{(2.3,0.3) }*+{{}^{\star}}="f4" !{(4.3,0.2) }*+{{}^{\star}}="f5"  !{(0.8,0.3) }*+{{}^{\star}}="f2"
   "f2"-"t4" "t3"-@{.}"t4" "f3"-@{~}"t6"
"t4"-"f4" "t7"-@{~}"f5" "t3"-@{~}^{1/2}"t7" "t3"-@{~}^{1/2}"t6"}}
\]
\caption{The $K$-vine of a branch. Here  $w=\nu_L(yc)$, in either case, and 
$v_{1/2}=B_\xi^{\big[w-t+\frac{1}{2}\big]}$ in (B).  Distances are normalized.}\label{figure 2} 
\end{figure}

We assume next that $m_q(X)\in K[X]$ is separable and irreducible. Let $L=K(\alpha)$ be the splitting field, where
$\alpha$ is a root of $m_q$. Define $w\in\matrici_2(L)$ as above.
Since $\text{Gal}(L/K)$ permutes transitively the ends of the path $\mathfrak{s}_L(w)$ (c.f. \cite[\S 5]{a-cb1}), it follows that  these ends are $x+y\alpha$ and $x+y(\alpha+c)$, for some $x,y \in K$. 
Formula (4.1) in \cite{a-cb1} tell us that the two idempotents satisfying 
$\mathfrak{s}_L(\tau)=\mathfrak{p}(a,b)_L$ are $\tau=\frac1{b-a}\sbmattrix{b}{-ab}{1}{-a}$ 
and $1-\tau$ (in any characteristic).  From here, a straightforward computation show that
the two possible choices for $q$ are $q=\beta(x,y,c,d)=\textnormal{\scriptsize{$ \frac{1}{y}
\left( \begin{array}{cc}x& x^2+cxy+dy^2  \\1 & x+yc \\ \end{array} \right)$\normalsize}}$ and
$q=\beta(x,y,c,d)+c$.  We define $\mathfrak{v}(\xi)=\mathfrak{p}(\xi, \infty)$, 
for every $\xi\in K$, and call it the $K$-vine of $\xi$.
The nearest $K$-vine $\mathfrak{v}(\xi)$ from $\mathfrak{s}_L(w)$, which we call the $K$-vine of $w$,
minimizes the expression $|x+\xi+y\alpha|=|x+\xi+y(\alpha+c)|$. 
See Fig. \ref{figure 2}\textbf{(A)} and Fig. \ref{figure 2}\textbf{(B)}.
In other words, $\xi \in K$ satisfies the identity 
$$ |x+\xi+y\alpha|^2= |x+\xi+y\alpha| |x+\xi+y(\alpha+c)|$$
$$= |yc|^2  \left| \left( \frac{x+\xi}{cy}\right)^2 + 
\left( \frac{x+\xi}{cy} \right) + \frac{d}{c^2} \right| 
= |yc|^2 \left| \mathbb{D} \left( \frac{d}{c^2} \right) \right| .$$ 
If $p_{\frac{d}{c^2}}$ is unramified, we have $|x+\xi+y\alpha|=|yc|$,
 by Theorem \ref{t21}. We conclude that the depth of 
$\mathfrak{s}_K(q)$ is $\nu(c)$, and its stem is the highest vertex in 
$\mathfrak{s}_L(w)$, as in Fig. \ref{figure 2}\textbf{(A)}. On the other hand, 
if $p_{\frac{d}{c^2}}$ is ramified, we have $|x+\xi+y\alpha|=
|yc|\cdot|\pi|^{-t+\frac12}$, where $\mathbb{D} \left( \frac{d}{c^2} \right)
=(\pi^{-2t+1})$ and $t>0$. Then, the stem of $\mathfrak{s}_K(q)$ is the line $\mathfrak{p}[v_0,v_1]$ in Fig. \ref{figure 2}\textbf{(B)}, and the depth of this branch is $\nu(c)-t$, which is non negative since $\nu\left(\frac{d}{c^2} \right) \leq -2t+1 $, so $\nu(c)-t \geq \frac{1}{2}(\nu(d)-1)\geq -\frac12$, as $d \in \mathcal{O}$. Note that, in the figure,
$v_0$ and $v_1$ are defined over $K$, but $v_{1/2}$ is not. 

Now, we assume that $m_q(X)= X^2+b$, and that there exist $\alpha \in \oink$ satisfying $\alpha^2=b.$ In this case $\eta= q+\alpha \in \mathbb{M}_2(K)$ is a nilpotent element. This implies that $\mathfrak{s}_K(q)
=\mathfrak{s}_K(\eta)$ is an infinite foliage (c.f. \cite[\S 2]{Eichler2}). Let $\mathfrak{f}$ be an infinite foliage whose end is $a \in \mathbb{P}^1(K)$.  If $a \in K$, there is a unique leaf $B \in V_{\mathfrak{f}}$ of the form $B=B_a^{[s]}$.
See Fig. \ref{figure 1}\textbf{(A)}. This is called the highest leaf of $\mathfrak{f}$. 
Note that every infinite foliage in the BTT is determined by its unique end and one leaf, so the infinite foliage described above is completely determined by the end $a$ and the number $s$, and it is denoted $\mathfrak{f}(a,s)$. When $a=\infty$, every leaf has the form $B=B_\varepsilon^{[s]}$, with $\varepsilon\in K$. In the latter case we say $\mathfrak{f}=\mathfrak{f}(\infty,s)$. See Fig. \ref{figure 1}\textbf{(A')}.

\begin{figure}
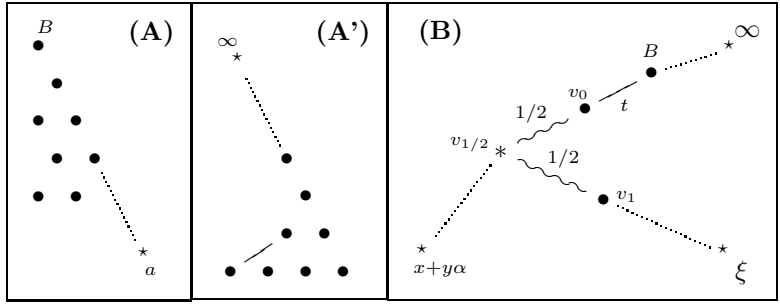

\[ \fbox{ \xygraph{
!{<0cm,0cm>;<0.5cm,0cm>:<0cm,1cm>::} 
!{(0,3) }*+{\bullet}="i" !{(0.2,3.2) }*+{{}^B}="in"
!{(0.5,2.5) }*+{\bullet}="b0" 
!{(1.0,2.0) }*+{\bullet}="b1" 
!{(1.5,1.5) }*+{\bullet}="b2" 
!{(0.5,1.5) }*+{\bullet}="c1" 
!{(0,1) }*+{\bullet}="d1" 
!{(1,1) }*+{\bullet}="d2" 
!{(0,2) }*+{\bullet}="c0"
!{(2.8,0.2) }*+{{}^{\star}}="f"
!{(3,-0.05) }*+{{}^{a}}="0on"
!{(3,3.18) }*+{\textnormal{\bf (A)}}="ti99"
"i"-"b0" "b1"-"b0" "c0"-"b0" "b1"-"b2" "c1"-"b1" "c1"-"d1" "c1"-"d2" "b2"-@{.}"f"
 } }
\fbox{ \xygraph{
!{<0cm,0cm>;<0.5cm,0cm>:<0cm,1.0cm>::} 
!{(3,0) }*+{\bullet}="i"
!{(2.5,0.5) }*+{\bullet}="b0" 
!{(2.0,1.0) }*+{\bullet}="b1" 
!{(1.5,1.5) }*+{\bullet}="b2" 
!{(1.5,0.5) }*+{\bullet}="c1" 
!{(1,0) }*+{\bullet}="d1" 
!{(0,0) }*+{\bullet}="d2" 
!{(2,0) }*+{\bullet}="c0"
!{(0.2,2.8) }*+{{}^{\star}}="f"
!{(-0.05,3) }*+{{}^{\infty}}="0on"
!{(3,3.18) }*+{\textnormal{\bf (A')}}="ti99"
!{(3,-0.15) }*+{\ }="ti99"
"i"-"b0" "b1"-"b0" "c0"-"b0" "b1"-"b2" "c1"-"b1" "c1"-"d1" "c1"-"d2" "b2"-@{.}"f"
 } }
\fbox{ \xygraph{
!{<0cm,0cm>;<.8cm,0cm>:<0cm,.8cm>::} 
 !{(0.5,0) }*+{{}^{x+y\alpha}}="e2" !{(5.6,4) }*+{\infty}="e3" !{(5.5,0) }*+{\xi}="e5"
!{(1.5,2) }*+{*}="t3" !{(1.0,2.1) }*+{{}^{v_{1/2}}}="t10"   !{(2.9,2.7) }*+{\bullet}="t4"  !{(3.2,1.2) }*+{\bullet}="t5" 
!{(4,3.3) }*+{\bullet}="t6" !{(4,3.6) }*+{{}^{B}}="t7" !{(2.8,2.9) }*+{{}^{v_0}}="t8"  !{(3.6,1.2) }*+{{}^{v_1}}="t9" 
!{(2.6,2) }*+{{}^{}}="ti3" !{(0.2,0.3) }*+{{}^{\star}}="f2" !{(5.2,0.3) }*+{{}^{\star}}="f5" !{(5.3,3.7) }*+{{}^{\star}}="f3"
!{(0.5,3,98) }*+{\textnormal{\bf (B)}}="ti99"
   "f2"-@{.}"t3"  "f3"-@{.}"t6" "t3"-@{~}^{1/2}"t5" "f5"-@{.}"t5" "t3"-@{~}^{1/2}"t4" "t6"-^{t}"t4" 
 } }
\]
\caption{The infinite foliage $\mathfrak{f}(a,s)$ for various values of $a$. In \textbf{(B)}, 
vertex $B$ is the highest leaf.}\label{figure 1}
\end{figure}

\begin{lemma}\label{lemma 4.1}
The set $ \mathbf{N}_{a,t}= \mathcal{O}^{*}  \textnormal{\scriptsize{$\left( \begin{array}{cc}
a\pi^{-t} & a^2\pi^{-t}  \\ \pi^{-t} & a \pi^{-t} \\ \end{array} \right)$\normalsize}}$
consists precisely of all nilpotent elements $\eta \in \mathbb{M}_2(K)$ satisfying $\mathfrak{s}_K(\eta)=\mathfrak{f}(a, t)$. On the other hand, the set 
$\mathbf{N}_{\infty,t}=\mathcal{O}^{*}\textnormal{\scriptsize{$\left( \begin{array}{cc}
0 & \pi^{t}  \\
0 & 0
 \\ \end{array} \right)$\normalsize}}$ consists precisely of all nilpotent elements $\eta \in \mathbb{M}_2(K)$ satisfying 
$\mathfrak{s}_K(\eta)= \mathfrak{f}(\infty, t)$. 
\end{lemma}
\begin{proof}
Assume first that $a \neq \infty$. Applying a Moebius transformation $\mu$ of the form $z \mapsto z+a$,
which corresponds to the matrix $\mu=\sbmattrix 1a01$.
 we can assume that $a=0$, since $\mu^{-1} \mathbf{N}_{a,t}\mu= \mathbf{N}_{0,t}$. In this case we have
$\eta \in \textnormal{\scriptsize{$\left( \begin{array}{cc}
\mathcal{O} & \pi^i\mathcal{O}  \\
\pi^{-i}\mathcal{O} & \mathcal{O}
 \\ \end{array} \right)$\normalsize}}$,
 for every $ i \geq t$, and $\eta \notin \textnormal{\scriptsize{$\left( \begin{array}{cc}
\mathcal{O} & \pi^{t-1}\mathcal{O}  \\
\pi^{-t+1}\mathcal{O} & \mathcal{O}
 \\ \end{array} \right) $\normalsize}}$. It follows that $\eta = \textnormal{\scriptsize{$\left( \begin{array}{cc}
0 & 0  \\u\pi^{-t} & 0 \\ \end{array} \right)$\normalsize}}$, for some $u \in \mathcal{O}^{*}$.
The case $a=\infty$ is analogous.
\end{proof}

We assume $m_q(X)$ is inseparable and irreducible. Let $L/K$ be the splitting field of $m_q(X)$ and 
let $\alpha \in L$ be a root. Observe that $\mathfrak{s}_K(q)$ is finite, as $\oink[q]$ 
is an integral domain (c.f. \cite[\S4]{Eichler2}). In particular, $\mathfrak{s}_L(q)$ is an infinite foliage 
whose unique end is $z=x+y\alpha$, as in Fig. \ref{figure 1}\textbf{(A)},  is in $L\backslash K$.
Assume $x \in K$ and  $y \in K^{*}$, so that 
$\mathfrak{s}_L(q)=\mathfrak{f}(z,s)$, i.e., $B=B_z^{[s]}$ is a leaf.

By Lemma \ref{lemma 4.1} we have $q=  \textnormal{\scriptsize{$\left( \begin{array}{cc}
(x+y\alpha)u\pi_L^{-s}+\alpha & (x^2+y^2b)u\pi_L^{-s}  \\u\pi_L^{-s} & (x+y\alpha)u\pi_L^{-s}+\alpha
 \\ \end{array} \right)$\normalsize}}$, for some $u \in \mathcal{O}_L^{*}$. As $q \in \mathbb{M}_2(K)$, 
we conclude that $u\pi_L^{-s}\in\mathcal{O}_K$, and therefore also that
 $y=u^{-1}\pi_L^s$, whence $q=\beta'(x,y,b)=\textnormal{\scriptsize{$ \frac{1}{y}
\left( \begin{array}{cc}x& x^2+y^2b  \\1 & x \\ \end{array} \right)$\normalsize}}$ 
and $B=B_z^{[\nu(y)]}$. On the other hand, the $K$-vine $\mathfrak{v}(\xi)$  of $\eta=q+\alpha$, 
which minimizes $|x+\xi+y\alpha|$, by definition satisfies 
$$|x+\xi+y\alpha|^2=|y|^2   \left| \left( \frac{x+\xi}{y}\right)^2+b \right| 
= |y|^2 \left| \delta \left(b \right) \right|.$$ 
Since $\delta(b)=(\pi^{2t+1})$, it follows that $|x+\xi+y\alpha|=|y|\cdot|\pi|^{t+\frac{1}{2}}$, where 
$t \in \mathbb{N}$. We conclude that the depth of $\mathfrak{s}_K(q)$ equals $t$, and its stem is the line 
$\mathfrak{p}[v_0,v_1]$ in Fig. \ref{figure 1}\textbf{(B)}, whose midpoint 
$v_{1/2}=B_\xi^{\big[v(y)+t+\frac{1}{2}\big]}$ is a ghost vertex.

\subsection{Computing symmetric products for split polynomials}\label{s5}

The purpose of this subsection is to provide Table 2, which allows us to easily compute the symmetric
product $\lambda=\Lambda(q_1,q_2)$ and the discriminant $\Delta$ from the branches of either quaternion 
in most cases. This often can be done up to some arbitrary units $u$ and $v$, but they are irrelevant in 
our main proofs.
\begin{table}
\[
\scalebox{0.8}[0.8]{
\begin{tabular}{|c|c|c|c|c|c|c|c|}
\hline $m_1$ & $ m_2$ &
$\mathfrak{m}_K(q_1)$ & $\mathfrak{m}_K(q_2)$ &$\lambda$&$\Delta$\\
\hline\hline
sep.   & sep. &  $\mathfrak{p}(0,\infty)$ & $\mathfrak{p}(1,\theta)$ & 
$a_1\alpha_2+a_2\alpha_1+\frac{a_1a_2}{1+\theta}$
& $\frac{a_1^2a_2^2\theta}{(1+\theta)^2}$\\ \hline
sep. & sep. &  $\mathfrak{p}(0,\infty)$ &  $\mathfrak{p}(0,\infty)$ & $a_1 \alpha_2+a_2 \alpha_1$&$0$\\ \hline
insep. & insep. &  $\mathfrak{f}(\infty,0)$ &  $\mathfrak{f}(0,s)$ & $uv\pi^{-s}$&$u^2v^2\pi^{-2s}$\\ \hline
insep. & insep. &  $\mathfrak{f}(\infty,0)$ &  $\mathfrak{f}(\infty,s)$  & $0$&$0$\\ \hline
sep. & insep. & $\mathfrak{p}(0,1)$   &  $\mathfrak{f}(\infty,r)$ & $a_1(\alpha_2+u\pi^r)$ 
&$a_1^2u^2\pi^{2r}$\\ \hline
sep. & insep. & $\mathfrak{p}(0,\infty)$   &  $\mathfrak{f}(\infty,r)$ & $a_1\alpha_2$&$0$\\ \hline
\end{tabular}}
\]
\caption{The value of $\nu(\lambda)$ for the particular branches described in \S4.1.
Here $u,v\in\mathcal{O}^*$ are arbitrary.}\label{tabla 2}
\end{table}
We assume everywhere that the polynomials $m_1,m_2 \in K[X]$ split, since we can otherwise compute in
a suitable extension where they do. Let $\alpha_i$ be a root of $m_i$, for $i\in\{1,2\}$, and assume the matrices 
$q_1,q_2 \in \mathbb{M}_2(K)$  satisfy $m_1(q_1)=m_2(q_2)=0$. Assume everywhere
that neither $q_1$ nor $q_2$ is a scalar.
 Recall that $\lambda=\Lambda(q_1,q_2)$, as defined in \eqref{133}, is a scalar matrix, 
and therefore a conjugation invariant. This
 allows us to chose a convenient conjugate of the pair $(q_1,q_2)$ to compute $\lambda$,
so in particular it allows us to apply Moebius transformations to move the stems to one of the pairs in the
table.
For simplicity we introduce, for any pair $(a,\alpha)\in\bar{K}\times\bar{K}$, the following matrices:
\small
$$
A(a,\alpha)= \left( \begin{array}{cc}
a+\alpha & 0 \\ 0 & \alpha \\ \end{array} \right), \, \,  A'(a, \alpha)= \left( \begin{array}{cc}
\alpha & a \\ 0 & \alpha \\ \end{array} \right), \, \,  A''(a, \alpha)= \left( \begin{array}{cc}
\alpha & 0 \\ a & \alpha \\ \end{array} \right).
$$
\normalsize

Assume that $m_1$ and $m_2$ are both separable. Let $w_i= \frac{q_i+\alpha_i}{a_i}\in\mathbb{M}_2(K)$,
for $i\in\{1,2\}$, be the corresponding idempotents. Note that  the quaternion involution satisfies
$\overline{q_i}=q_i+a_i$, $\overline{w_i}=\frac{\overline{q_i}+\alpha_i}{a_i}=1+w_i$ and
$\Lambda(\overline{q_1},q_2)=\Lambda(q_1,\overline{q_2})=\lambda+a_1a_2$. For any idempotent $w$,
the only other idempotent with the same branch is $\overline{w}=1+w$. There is, therefore, some ambiguity
in the choice of $q_1$ and $q_2$, but this coincide, in fact, with the ambiguity in the choice of the 
roots in the expression for $\lambda$. As 
$\Delta(\lambda, m_1, m_2)= \Delta(\lambda+ a_1a_2, m_1,m_2)$, this has no effect on the computation of 
$\Delta$.

If the stems $\mathfrak{s}_K(w_1)$ and $\mathfrak{s}_K(w_2)$ fail to coincide, we can assume that they are
the maximal paths $\mathfrak{s}_K(w_1)=\mathfrak{p}(0,\infty)$ and 
$\mathfrak{s}_K(w_2)=\mathfrak{p}(1,\theta)$, 
respectively (c.f. \cite[\S 4]{a-cb1}), where $\theta$ could be $\infty$. As in the preceding reference,
 we can assume  
$w_1=\textnormal{\scriptsize{$\left( \begin{array}{cc} 1 & 0 \\ 0 & 0  \end{array} \right)$\normalsize}}$ and 
$w_2= \frac{1}{\theta+1}\textnormal{\scriptsize{$\left( \begin{array}{cc} \theta & \theta \\ 
1 & 1  \end{array} \right)$\normalsize}}$. 
This leads to the choice $q_1=A(a_1,\alpha_1)$ and $q_2=\left( \begin{array}{cc}
\frac{a_2 \theta}{1+\theta} + \alpha_2 & \frac{a_2 \theta}{1+\theta} \\
\frac{a_2}{1+\theta} & \frac{a_2}{1+\theta} + \alpha_2  \end{array} \right)$. On the other hand,
if $\mathfrak{s}_K(w_1) = \mathfrak{s}_K(w_2)$, we can assume this branch is the maximal path 
$\mathfrak{p}(0, \infty)$. In the latter case we assume 
$w_1=w_2=\textnormal{\scriptsize{$\left( \begin{array}{cc} 1 & 0 \\ 0 & 0  \end{array} \right)$\normalsize}}$, 
so that $q_1=A(a_1,\alpha_1)$ and $q_2=A(a_2, \alpha_2)$. Now a straightforward computation gives us the first
two lines of the table.

Now, we assume that $m_1$ and $m_2$ are both inseparable. Let $\eta_i= q_i+\alpha_i\in \mathbb{M}_2(K)$,
for $i\in\{1,2\}$,  be the associated nilpotent elements. Moebius transformations act transitively on pairs 
$(a,b) \in \mathbb{P}^1(K) \times \mathbb{P}^1(K)$ with $a\neq b$. Therefore, applying a 
Moebius transformation if we needed, we can assume that the ends of $\mathfrak{s}_K(\eta_1)$ and $\mathfrak{s}_K(\eta_2)$ are $0$ and $\infty$, unless they coincide. In the latter case we assume both to
be $\infty$. Hence, we assume $\mathfrak{s}_k(\eta_1)=\mathfrak{f}(\infty, 0)$, in either case, while we
assume  either $\mathfrak{s}_k(\eta_2)= \mathfrak{f}(0,s)$ or $\mathfrak{s}_K(\eta_2)= \mathfrak{f}(\infty,s)$,
according to the case. Equivalently, we assume $\eta_1=A'(u,0)$, and we assume either 
$\eta_2= A''(v\pi^{-s},0)$ or $\eta_2= A'(v\pi^s,0)$, for some $u,v \in \mathcal{O}^{*}$. 
This gives us the third and fourth lines.

Finally, assume that $m_1$ is separable and $m_2$ is inseparable. Let $w_1= \frac{q_1+\alpha_1}{a_1}$ as before.
 Suppose first that the end of $\mathfrak{s}_K(q_2)$ is not an end of $\mathfrak{s}_K(w_1)$. Then, as Moebius transformations act transitively on triples of different ends, we can assume 
$\mathfrak{s}_K(q_2) = \mathfrak{f}(\infty,r)$ and $\mathfrak{s}_K(w_1)= \mathfrak{p}(1,0)$. 
This implies that 
$q_1 =\textnormal{\scriptsize{$\left( \begin{array}{cc} \alpha_1 & 0 \\ 
a_1 & a_1+\alpha_1  \end{array} \right)$\normalsize}}$ 
and $q_2= A'(u\pi^r,\alpha_2)$, for some $u \in \oink^{*}$. On the other hand, if $\mathfrak{s}_K(w_1)$ and $\mathfrak{s}_K(q_2)$ have a common end, then we can choose $q_2$
 as before, while we assume $\mathfrak{s}_K(w_1)=\mathfrak{p}(0,\infty)$, and hence 
$q_1 =A(a_1, \alpha_1)$. The result follows.

\section{Proof of Theorem \ref{t22} in the inseparable case}\label{i}

We denote, as usual, the depth and diameter of the branch $\mathfrak{u}$ (of an order) by 
$p(\mathfrak{u})$ and $d(\mathfrak{u})$, respectively (c.f. \cite{osbotbtt}). 
In all of \S5, $m_i$, $q_i$, $\alpha_i$ and $\lambda$ are as in \S4.1, except that the polynomials
$m_1$ and $m_2$ are assumed to be inseparable, whence $a_1=a_2=0$ in the whole section. 
Denote by $\eta_i=q_i+\alpha_i$ the associated nilpotent for either value of $i\in\{1,2\}$. 
We make frequent use of next result: 
\begin{lemma}\cite[Prop. 2.4]{osbotbtt} \label{lema de int de hojas}
Let $\mathfrak{f}_1,\mathfrak{f}_2$ be two infinite foliages, and let $d= d(\mathfrak{f}_1\cap \mathfrak{f}_2)$. Then, if $d<\infty$, we have $p(\mathfrak{f}_1 \cap \mathfrak{f}_2) = 
\lfloor \frac{d}{2} \rfloor$ and the stem of $\mathfrak{f}_1 \cap \mathfrak{f}_2$ is a vertex when
$d$ is even and an edge otherwise. If $d=\infty$, then $\mathfrak{f}_1 \subset \mathfrak{f}_2$ 
or $\mathfrak{f}_2 \subset \mathfrak{f}_1$.
\end{lemma}

\begin{figure}
\[
\unitlength 1mm 
\linethickness{0.4pt}
\ifx\plotpoint\undefined\newsavebox{\plotpoint}\fi 
\begin{picture}(100.5,32.25)(8,0)
\put(8,-2){\framebox(50,37)[c]{ }}
\put(58,-2){\framebox(50,37)[c]{ }}
\multiput(35.43,23.43)(0,-.98611){19}{{\rule{.4pt}{.4pt}}}
\put(35.5,22.5){\line(0,1){6.25}}
\put(35.5,7){\line(0,-1){4}}
\put(37,31.5){\makebox(0,0)[cc]{$\star^{\infty}$}}
\put(36.5,.5){\makebox(0,0)[cc]{$\star_0$}}
\put(28.43,22.43){\line(1,0){.9868}}
\put(30.403,22.43){\line(1,0){.9868}}
\put(32.377,22.43){\line(1,0){.9868}}
\put(34.351,22.43){\line(1,0){.9868}}
\put(36.324,22.43){\line(1,0){.9868}}
\put(38.298,22.43){\line(1,0){.9868}}
\put(40.272,22.43){\line(1,0){.9868}}
\put(42.245,22.43){\line(1,0){.9868}}
\put(44.219,22.43){\line(1,0){.9868}}
\put(46.193,22.43){\line(1,0){.9868}}
\put(28.68,6.93){\line(1,0){.9674}}
\put(30.614,6.93){\line(1,0){.9674}}
\put(32.549,6.93){\line(1,0){.9674}}
\put(34.484,6.93){\line(1,0){.9674}}
\put(36.419,6.93){\line(1,0){.9674}}
\put(38.354,6.93){\line(1,0){.9674}}
\put(40.288,6.93){\line(1,0){.9674}}
\put(42.223,6.93){\line(1,0){.9674}}
\put(44.158,6.93){\line(1,0){.9674}}
\put(46.093,6.93){\line(1,0){.9674}}
\put(48.028,6.93){\line(1,0){.9674}}
\put(49.962,6.93){\line(1,0){.9674}}
\put(35.5,26.5){\line(-4,-1){6}}
\multiput(29.5,25)(-.03947368,-.03289474){38}{\line(-1,0){.03947368}}
\multiput(30,25)(.05,-.0333333){30}{\line(1,0){.05}}
\multiput(35.5,24.75)(-.1,-.0333333){30}{\line(-1,0){.1}}
\put(35.75,4){\line(-4,-1){6}}
\multiput(29.75,2.5)(-.03947368,-.03289474){38}{\line(-1,0){.03947368}}
\multiput(30.25,2.5)(.05,-.0333333){30}{\line(1,0){.05}}
\multiput(35.75,5.75)(-.1,-.0333333){30}{\line(-1,0){.1}}
\multiput(35.5,26.5)(.1916667,-.0333333){30}{\line(1,0){.1916667}}
\multiput(35.75,24.5)(.15,-.0333333){15}{\line(1,0){.15}}
\multiput(41,25.25)(-.03947368,-.03289474){38}{\line(-1,0){.03947368}}
\multiput(41,25.25)(.03947368,-.03289474){38}{\line(1,0){.03947368}}
\multiput(35.75,3.75)(.1916667,-.0333333){30}{\line(1,0){.1916667}}
\multiput(36,5.5)(.15,-.0333333){15}{\line(1,0){.15}}
\multiput(41.25,2.5)(-.03947368,-.03289474){38}{\line(-1,0){.03947368}}
\multiput(41.25,2.5)(.03947368,-.03289474){38}{\line(1,0){.03947368}}
\put(20,23){\makebox(0,0)[cc]{\scriptsize{Level $0$}}}
\put(19,7){\makebox(0,0)[cc]{\scriptsize{Level $s<0$}}}
\put(85,22.5){\line(0,1){6.25}}
\put(85,7){\line(0,-1){4}}
\put(87,31){\makebox(0,0)[cc]{$\star^{\infty}$}}
\put(86,0){\makebox(0,0)[cc]{$\star_0$}}
\put(77.93,22.43){\line(1,0){.9868}}
\put(79.903,22.43){\line(1,0){.9868}}
\put(81.877,22.43){\line(1,0){.9868}}
\put(83.851,22.43){\line(1,0){.9868}}
\put(85.824,22.43){\line(1,0){.9868}}
\put(87.798,22.43){\line(1,0){.9868}}
\put(89.772,22.43){\line(1,0){.9868}}
\put(91.745,22.43){\line(1,0){.9868}}
\put(93.719,22.43){\line(1,0){.9868}}
\put(95.693,22.43){\line(1,0){.9868}}
\put(78.18,6.93){\line(1,0){.9674}}
\put(80.114,6.93){\line(1,0){.9674}}
\put(82.049,6.93){\line(1,0){.9674}}
\put(83.984,6.93){\line(1,0){.9674}}
\put(85.919,6.93){\line(1,0){.9674}}
\put(87.854,6.93){\line(1,0){.9674}}
\put(89.788,6.93){\line(1,0){.9674}}
\put(91.723,6.93){\line(1,0){.9674}}
\put(93.658,6.93){\line(1,0){.9674}}
\put(95.593,6.93){\line(1,0){.9674}}
\put(97.528,6.93){\line(1,0){.9674}}
\put(99.462,6.93){\line(1,0){.9674}}
\put(85,26.5){\line(-4,-1){6}}
\multiput(79,25)(-.03947368,-.03289474){38}{\line(-1,0){.03947368}}
\multiput(79.5,25)(.05,-.0333333){30}{\line(1,0){.05}}
\multiput(85,24.75)(-.1,-.0333333){30}{\line(-1,0){.1}}
\put(85.25,6){\line(-4,-1){6}}
\multiput(79.25,4.5)(-.03947368,-.03289474){38}{\line(-1,0){.03947368}}
\multiput(79.75,4.5)(.05,-.0333333){30}{\line(1,0){.05}}
\multiput(85.25,4.25)(-.1,-.0333333){30}{\line(-1,0){.1}}
\multiput(85,26.5)(.1916667,-.0333333){30}{\line(1,0){.1916667}}
\multiput(85.25,24.5)(.15,-.0333333){15}{\line(1,0){.15}}
\multiput(90.5,25.25)(-.03947368,-.03289474){38}{\line(-1,0){.03947368}}
\multiput(90.5,25.25)(.03947368,-.03289474){38}{\line(1,0){.03947368}}
\multiput(85.25,5.75)(.1916667,-.0333333){30}{\line(1,0){.1916667}}
\multiput(85.5,3.75)(.15,-.0333333){15}{\line(1,0){.15}}
\multiput(90.75,4.5)(-.03947368,-.03289474){38}{\line(-1,0){.03947368}}
\multiput(90.75,4.5)(.03947368,-.03289474){38}{\line(1,0){.03947368}}
\put(68,23){\makebox(0,0)[cc]{\scriptsize{Level $s>0$}}}
\put(72,7){\makebox(0,0)[cc]{\scriptsize{Level $0$}}}
\thicklines
\put(84.75,22.5){\line(0,-1){15.5}}
\put(85.5,22){\line(0,-1){14.25}}
\multiput(85,20.75)(-.06111111,-.03333333){45}{\line(-1,0){.06111111}}
\multiput(83,19)(.0652174,.0326087){23}{\line(1,0){.0652174}}
\multiput(85.75,20.75)(.03947368,-.03289474){38}{\line(1,0){.03947368}}
\multiput(85.25,20)(.05,-.0333333){30}{\line(1,0){.05}}
\multiput(85.5,18.5)(.1583333,-.0333333){30}{\line(1,0){.1583333}}
\multiput(85.5,17.75)(.28125,-.03125){8}{\line(1,0){.28125}}
\put(87.75,17.5){\line(0,-1){1.75}}
\multiput(88.5,17.25)(.125,-.03125){8}{\line(1,0){.125}}
\put(88.5,17){\line(0,-1){1}}
\multiput(85.5,12.75)(.1583333,-.0333333){30}{\line(1,0){.1583333}}
\multiput(85.5,12)(.28125,-.03125){8}{\line(1,0){.28125}}
\put(87.75,11.75){\line(0,-1){1.75}}
\multiput(88.5,11.5)(.125,-.03125){8}{\line(1,0){.125}}
\put(88.5,11.25){\line(0,-1){1}}
\multiput(86,9.5)(.03947368,-.03289474){38}{\line(1,0){.03947368}}
\multiput(85.5,8.75)(.05,-.0333333){30}{\line(1,0){.05}}
\multiput(84.75,12.75)(-.11842105,-.03289474){38}{\line(-1,0){.11842105}}
\multiput(85,12)(-.3125,-.03125){8}{\line(-1,0){.3125}}
\multiput(82.5,11.75)(.0326087,-.0652174){23}{\line(0,-1){.0652174}}
\multiput(81,11)(.15625,.03125){8}{\line(1,0){.15625}}
\multiput(82.25,11.25)(.03125,-.125){8}{\line(0,-1){.125}}
\multiput(84.25,18)(-.11842105,-.03289474){38}{\line(-1,0){.11842105}}
\multiput(84.5,17.25)(-.3125,-.03125){8}{\line(-1,0){.3125}}
\multiput(82,17)(.0326087,-.0652174){23}{\line(0,-1){.0652174}}
\multiput(80.5,16.25)(.15625,.03125){8}{\line(1,0){.15625}}
\multiput(81.75,16.5)(.03125,-.125){8}{\line(0,-1){.125}}
\multiput(85.25,9.75)(-.06111111,-.03333333){45}{\line(-1,0){.06111111}}
\multiput(83.25,8)(.0652174,.0326087){23}{\line(1,0){.0652174}}
\put(46.25,31.25){\makebox(0,0)[cc]{\textbf{(A)}}}
\put(96.5,32.25){\makebox(0,0)[cc]{\textbf{(B)}}}
\end{picture}
\]
\caption{Distance between infinite foliages with ends $0$ and $\infty$ \textbf{(A)}, or the
diameter of the intersection \textbf{(B)}.} \label{figure 3a}
\end{figure}
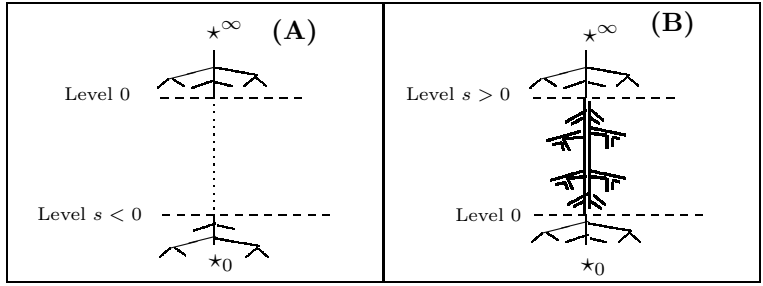
If $m_1$ has a root $\alpha_1 \in K$, we assume $\eta_1=q_1+\alpha_1=
\textnormal{\scriptsize{$\left( \begin{array}{cc} 0 & 1 \\ 0 & 0  \end{array} \right)$\normalsize}}$, whence $\mathfrak{s}_K(q_1)=\mathfrak{s}_K(\eta_1)=\mathfrak{f}(\infty,0)$. 
Furthermore, if both $m_1$ and $m_2$ have roots in $K$, and if the branches $\mathfrak{s}_K(q_1)$ and $\mathfrak{s}_K(q_2)$ have different ends, we can assume also $\mathfrak{s}_K(q_2)=\mathfrak{f}(0,s)$. Then, if $\mathfrak{s}_K(q_1) \cap \mathfrak{s}_K(q_2) = \emptyset$, we have $s>0$ and the distance between these branches is $s=-\nu(\lambda)=-\frac12\nu(\Delta)$, by Table \ref{tabla 2}, which equals $d_f$ by Table \ref{table 1}. 
See  Fig. \ref{figure 3a}\textbf{(A)}.
 On the other hand, if $\mathfrak{u}=\mathfrak{s}_K(q_1) \cap \mathfrak{s}_K(q_2)$ in not empty, then
its diameter is $d(\mathfrak{u})=\nu(\lambda)$ (see Fig. \ref{figure 3a}\textbf{(B)}), so Lemma 
\ref{lema de int de hojas} implies that $p(\mathfrak{u})= \lfloor \frac{\nu(\lambda)}{2} \rfloor$ 
and the stem of $\mathfrak{u}$ has one or two vertices
 depending on the parity of $\nu(\lambda)$. From Table \ref{tabla 2}, we know that 
$\mathfrak{s}_K(q_1)$ and $\mathfrak{s}_K(q_2)$ have the same end precisely when $\lambda=0$. 
In the latter case one branch is contained in the other.

Now we assume that $m_1$ has a root in $K$ but $m_2$ does not. In this case we still assume that
$\mathfrak{s}_K(q_1)=\mathfrak{f}(\infty,0)$.  Let $L=K(\sqrt{b_2})$. We denote by 
$\mu_2=x+y\sqrt{b_2}$ the unique end of $\mathfrak{s}_L(q_2)=\mathfrak{f}\Big(\mu_2,\nu(y)\Big)$ 
and set $\delta\left( b_2\right) = (\pi^{2t_2+1})$, where $t_2 \in \mathbb{Z}$. Note that, in this case,
the $K$-stem $\mathfrak{m}_K(q_2)$ is the edge with endpoints $v_0$ and $v_1$ in Fig. 2 (\textbf{B}).
By applying a Moebius transformation, namely a translation taking $\mu_2$ to $0$, 
we are able to use Table 2 to conclude that $\frac12\nu(\Delta)=\nu(\lambda)$
is the level of the highest leaf of $\mathfrak{s}_L(q_2)$, which according to Fig. 2(\textbf{B}) is $-\nu(y)$.
 In this case, if 
$\mathfrak{m}_K(q_1) \cap \mathfrak{m}_K(q_2) = \emptyset$, the distance between these branches is $\nu(y)+t_2=-\frac12\nu(\Delta)+t_2=d_f$, since the stem $\mathfrak{m}_K(q_2)$ is a distance $t_2$ 
from every leaf of $\mathfrak{s}_L(q_2)$ (see Fig. 2\textbf{(B)} again). For the same reason, if 
$\mathfrak{m}_K(q_1) \cap \mathfrak{m}_K(q_2)  \neq \emptyset$, this intersection is precisely one vertex,
namely $v_0$,  if $t_2=-\nu(y)$, while 
$\mathfrak{s}_K(q_2) \subset \mathfrak{s}_K(q_1)$ if $t_2<-\nu(y)$.
The last two conditions are equivalent to $d_f=0$ and $d_f<0$ respectively, so the result follows in this case.

Finally, we assume that $m_1$ and $ m_2$ are both irreducible. Let $L=k(\sqrt{b_1}, \sqrt{b_2})$. We denote by 
$\mu_i= x_i+y_i\sqrt{b_i}$ the end of $\mathfrak{s}_L(q_i)$, and we define the fake branch
 $\mathfrak{f}(q_i)=\mathfrak{s}_L(\pi^{-t_i-1}\eta_i)=\mathfrak{s}_L(\pi^{-t_i-1}q_i)$, for $i\in\{1,2\}$.
 Note that $\mathfrak{f}(q_i)$ is at distance $t_i+1$ from the leaves of $\mathfrak{s}_L(q_i)$.
 In fact, a leaf of each fake branch, denoted by a white circle in Fig. \ref{figure 3b}, is at a distance $1/2$ from the 
midpoint (denoted by an asterisk) of the corresponding $K$-stem. Now assume that the stems $\mathfrak{m}_K(q_1)$ 
and $\mathfrak{m}_K(q_2)$ are different. Then, applying the computation in the first case over the extension $L$, 
and noting that 
$\Lambda(\pi^{-t_1-1}q_1,\pi^{-t_2-1}q_2)=\pi^{-t_1-t_2-2}\lambda$, we
obtain $d(\mathfrak{f}(q_1), \mathfrak{f}(q_2))=-\nu(\lambda)+t_1+t_1+2$, and this distance is at least $2$,
as we can see in Fig. 4, where the possible relative position of the $K$-stems are depicted. 
Now, as the distance between a leaf of the stem of $\mathfrak{s}_K(q_i)$ and $\mathfrak{f}(q_i)$ is $1$,
for $i\in\{1,2\}$, if $\mathfrak{s}_K(q_1)$ fails to intersect $\mathfrak{s}_K(q_2)$, 
then $d(\mathfrak{m}_K(q_1), \mathfrak{m}_K(q_2))=-\nu(\lambda)+t_1+t_1=d_f$ (Fig. \ref{figure 3b}\textbf{(A)}), while they intersect in a single point precisely when $d_f=0$, as in Fig. \ref{figure 3b}\textbf{(B)}. When the stems coincide, both fake branches have a vertex at distance $1/2$ from the common midpoint. 
Note that this is, for either branch, the ghost vertex $v_{1/2}$ in Fig. 2\textbf{(B)}. 
We conclude that either $d(\mathfrak{f}(q_1), \mathfrak{f}(q_2))=1$, as in Fig. \ref{figure 3b}\textbf{(C)}, or $\mathfrak{f}(q_1)$ intersects $\mathfrak{f}(q_2)$ non-trivially. In either case $-\nu(\lambda)+t_1+t_1+2<2$,
whence $d_f=\nu(\lambda)+t_1+t_1<0$. The result follows.

\begin{figure}
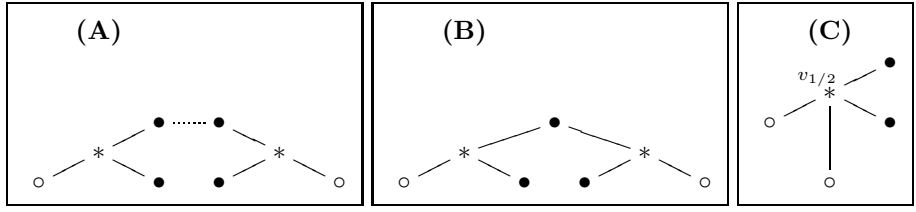

\[ 
\fbox{ \xygraph{
!{<0cm,0cm>;<.8cm,0cm>:<0cm,.8cm>::}
 !{(0,5) }*+{\textbf{(A)}}="na"
 !{(0,3) }*+{*}="x1" !{(-1,2.5) }*+{\circ}="01" !{(1,3.5) }*+{\bullet}="pa1" !{(1,2.5) }*+{\bullet}="pb1"
 !{(3,3) }*+{*}="x2" !{(4,2.5) }*+{\circ}="02" !{(2,3.5) }*+{\bullet}="pa2" !{(2,2.5) }*+{\bullet}="pb2"
"x1"-"01" "x1"-"pa1" "x1"-"pb1" "x2"-"02" "x2"-"pa2" "x2"-"pb2" "pa1"-@{.}"pa2"
 } }\ 
\fbox{ \xygraph{
!{<0cm,0cm>;<.8cm,0cm>:<0cm,.8cm>::}
 !{(0,5) }*+{\textbf{(B)}}="na"
 !{(0,3) }*+{*}="x1" !{(-1,2.5) }*+{\circ}="01" !{(1.5,3.5) }*+{\bullet}="pa1" !{(1,2.5) }*+{\bullet}="pb1"
 !{(3,3) }*+{*}="x2" !{(4,2.5) }*+{\circ}="02"  !{(2,2.5) }*+{\bullet}="pb2"
"x1"-"01" "x1"-"pa1" "x1"-"pb1" "x2"-"02" "x2"-"pa1" "x2"-"pb2" 
 } }\ 
\fbox{ \xygraph{
!{<0cm,0cm>;<.8cm,0cm>:<0cm,.8cm>::}
 !{(0,4) }*+{\textbf{(C)}}="na"
 !{(0,3) }*+{*}="x1" !{(-0.2,3.2) }*+{{}^{v_{1/2}}}="x1n"
!{(-1,2.5) }*+{\circ}="01" !{(1,3.5) }*+{\bullet}="pa1" !{(1,2.5) }*+{\bullet}="pb1"
 !{(0,1.5) }*+{\circ}="02"  
"x1"-"01" "x1"-"pa1" "x1"-"pb1" "x1"-"02" 
 } }
\]
\caption{Relative position of the stems for two irreducible inseparable polynomials.} \label{figure 3b}
\end{figure}

\section{Proof of Theorem \ref{t22} in the separable case}\label{s}

In all of this section, we assume $m_1$ and $m_2$ are separable polynomials. We let $q_i$, $\alpha_i$,
$\lambda$ and $\Delta$ be as in \S4.1. As before, we first prove Theorem \ref{t22} in the case
where  both polynomials split over $K$. For $i\in\{1,2\}$, let $w_i = \frac{q_i+\alpha_i}{a_i}$ be the 
associated idempotent.

From the definition of the fake distance $d_f$  in Th. \ref{t22}, we observe that
$d_f=-\infty$ precisely when $\Delta$ is $0$.
We need to prove that this is the case precisely when  the length $l$ of the intersection 
$\mathfrak{m}_K(q_1) \cap \mathfrak{m}_K(q_2)$ is 
not finite. Furthermore, we need to prove that 
$l=2\infty$ if and only if $q_1$ and $q_2$ commute, while $l=\infty$ when they fail to do so.
Now, the value $2\infty$ is attained precisely when both stems coincide. 
 By the arguments given in \S \ref{s5}, this is the case precisely when one of the following alternatives
hold:
\begin{itemize}
\item $w_1=w_2$ or
\item  $w_1=\overline{w_2}=1+w_2$. 
\end{itemize}
As a commutative sub-algebra contains at most $2$ idempotents, the condition``$q_1$ and $q_2$ commute" 
is equivalent to $l=2\infty$. In this case, we do have $l=2\infty=\mathrm{max}\{-2d_f,l(q_1),l(q_2)\}$,
with the natural conventions. In the remainder of the split case we assume that $q_1$ and $q_2$ fail to 
commute.

We assume that the stems of the quaternions $q_1,q_2 \in \mathbb{M}_2(K)$ 
are as shown in the first row of Table \ref{tabla 2}, i.e., 
$\mathfrak{m}_K(w_1)=\mathfrak{p}(0,\infty)$ and $\mathfrak{m}_K(w_2)=\mathfrak{p}(1,\theta)$,
with $\theta\notin\{0,1\}$. 
In this case, we have $\frac{\Delta}{a_1^2a_2^2}=\frac{\theta}{(1+\theta)^2}$.
 If $\theta=\infty$, we have $d_f=-\infty$, and also the stems intersect in a ray, as shown in 
Fig. \ref{figure 4}\textbf{(C)}. This takes care of the last infinite case.
 Suppose next that $\theta \in K$. Assume first that 
$\mathfrak{m}_K(q_1)\cap \mathfrak{m}_K(q_2) = \emptyset$.
 Then, as seen in Fig. \ref{figure 4}\textbf{(A)}, the smallest ball containing both $0$ and $\theta$ is
$B_0^{[0]}$, whence    $\nu(\theta)=0$ and
$d\Big(\mathfrak{m}_K(q_1),\mathfrak{m}_K(q_2)\Big)= \nu(1+\theta)= 
-\frac{1}{2}\nu\left(  \frac{\theta}{(1+\theta)^2}\right) = d_f$.
 If $\mathfrak{m}_K(q_1)\cap \mathfrak{m}_K(q_2) \neq \emptyset$, 
using the Moebius transformation $\sigma(z) = z^{-1}$ if needed, we can assume $\theta \in \mathcal{O}$, as in 
Fig. \ref{figure 4}\textbf{(B)}. Note that $\frac \theta{(1+\theta)^2} = 
\frac{\sigma(\theta)}{(1+\sigma(\theta))^2}$ is invariant unther this transformation. 
Then the length of $\mathfrak{m}_K(q_1)\cap \mathfrak{m}_K(q_2)$ is $l=\nu(\theta)= \nu\left(  \frac{\theta}{(1+\theta)^2}\right)=-2d_f$, and the result follows.

\begin{figure}
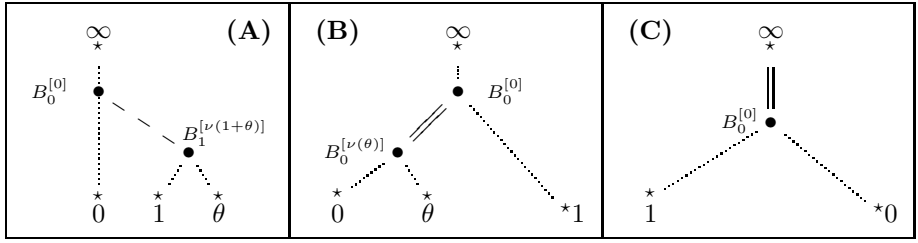

\[ 
\fbox{ \xygraph{
!{<0cm,0cm>;<.8cm,0cm>:<0cm,.8cm>::} 
 !{(1,0) }*+{0}="e2" !{(1,3) }*+{\infty}="e3"
!{(2,0) }*+{1}="e4" !{(3,0) }*+{\theta}="e5"
  !{(2.5,1) }*+{\bullet}="t3"   !{(1,2) }*+{\bullet}="t4" 
!{(3.1,1.3) }*+{{}^{B_1^{[\nu(1+\theta)]}}}="ti3"
!{(0.2,2) }*+{{}^{B_0^{[0]}}}="ti4"
!{(3.5,3) }*+{\textbf{(A)}}="ti99"
 !{(1,0.2) }*+{{}^{\star}}="f2" !{(1,2.7) }*+{{}^{\star}}="f3"
!{(2,0.2) }*+{{}^{\star}}="f4" !{(3,0.2) }*+{{}^{\star}}="f5"
   "f2"-@{.}"f3" "t3"-@{--}"t4"
"t3"-@{.}"f4" "t3"-@{.}"f5"
 } }
\fbox{ \xygraph{
!{<0cm,0cm>;<.8cm,0cm>:<0cm,.8cm>::} 
 !{(0.5,0) }*+{0}="e2" !{(2.5,3) }*+{\infty}="e3"
!{(2,0) }*+{\theta}="e4" !{(4.5,0) }*+{1}="e5"
  !{(2.5,2) }*+{\bullet}="t3"   !{(1.5,1) }*+{\bullet}="t4" 
!{(3.3,2) }*+{{}^{B_0^{[0]}}}="ti3"
!{(0.8,1) }*+{{}^{B_0^{[\nu(\theta)]}}}="ti4"
!{(0.5,3) }*+{\textbf{(B)}}="ti99"
 !{(0.5,0.25) }*+{{}^{\star}}="f2" !{(2.5,2.7) }*+{{}^{\star}}="f3" !{(2,0.2) }*+{{}^{\star}}="f4" !{(4.3,0) }*+{{}^{\star}}="f5"
   "f2"-@{.}"t4" "t3"-@{=}"t4" "f3"-@{.}"t3"
"t4"-@{.}"f4" "t3"-@{.}"f5"
 } }
\fbox{ \xygraph{
!{<0cm,0cm>;<.8cm,0cm>:<0cm,.8cm>::} 
 !{(4.5,0) }*+{0}="e2" !{(2.5,3) }*+{\infty}="e3"
 !{(0.5,0) }*+{1}="e5"
  !{(2.5,1.5) }*+{\bullet}="t3"  
!{(2.0,1.5) }*+{{}^{B_0^{[0]}}}="ti3"
!{(0.5,3) }*+{\textbf{(C)}}="ti99"
 !{(0.5,0.25) }*+{{}^{\star}}="f2" !{(2.5,2.7) }*+{{}^{\star}}="f3" !{(4.3,0) }*+{{}^{\star}}="f5"
  "t3"-@{.}"f2" "f3"-@{=}"t3" "t3"-@{.}"f5"
 } }
\]
\caption{The different possible configurations for the paths in the split separable case (\S6).}\label{figure 4}
\end{figure}

Now we handle the cases where one or both roots  $\alpha_1$ or $\alpha_2$ is not in $K$.
 In any of the following cases:
\begin{itemize}
 \item $m_1$ splits and $m_2$ is unramified,
 \item $m_2$ splits and $m_1$ is unramified or
 \item both $m_1$ and $m_2$ are unramified
\end{itemize}
the fake distance coincide with the the distance between the $L$-stems. Here we have the same cases depicted
in  \cite[Fig. 4]{a-cb1}. We can reason exactly as in \cite[\S6]{a-cb1} to finish the proof in this case.

Now, we assume that $m_1$ splits and $m_2$ is ramified. Then, by the split case, we know that the distance 
between the $L$-stems $\mathfrak{m}_L(q_1)$ and $\mathfrak{m}_L(q_2)$, where $L$ is the decomposition 
field of $m_2$, is $-\frac{1}{2}\nu\left( \frac{\Delta}{a_1^2 a_2^2}\right)$. The $K$-stem
$\mathfrak{m}_K(q_1)$ coincide with $\mathfrak{m}_L(q_1)$, under the usual identification,
and it is a maximal path in $\mathfrak{t}=\mathfrak{t}(K)$,
so it cannot be closer to $\mathfrak{m}_L(q_2)$ than the corresponding $K$-vine, which contains
$\mathfrak{m}_K(q_2)$, as in Fig. 1\textbf{(B)}. As the endpoints of the stem $\mathfrak{m}_K(q_2)$ are at 
distance $t_2$ from the stem $\mathfrak{m}_L(q_2)$, where $\mathbb{D}\left(\frac{b_2}{a_2^2} \right) =(\pi^{-2t_2+1})$, as seen in Fig. \ref{figure 2}\textbf{(B)}, if the stem $\mathfrak{m}_K(q_1)$ 
does not contain the stem $\mathfrak{m}_K(q_2)$, then the distance between these stems is 
$ -\frac{1}{2}\nu\left( \frac{\Delta}{a_1^2 a_2^2}\right)-t_2= d_f$, and the intersection  is
a vertex precisely when $d_f$ is $0$. Let $\alpha_1$ be a root of $m_1$. Note that 
\small
$$\frac{\Delta}{a_1^2a_2^2} = 
\left( \frac{\lambda}{a_1a_2}+\alpha_1\right) ^2 + 
\left( \frac{\lambda}{a_1a_2}+\alpha_1\right) + \frac{b_2}{a_2^2} ,$$
\normalsize
satisfies $\nu\left( \frac{\Delta}{a_1^2a_2^2}\right) \leq \nu \left( \mathbb{D} \left( \frac{b_2}{a_2^2} \right) \right)
= 2t_2+1$, whence $d_f= -\frac{1}{2}\nu\left( \frac{\Delta}{a_1^2a_2^2}\right) - t_2\geq-1/2$.
 Moreover, $\mathfrak{m}_K(q_1)$ contains $\mathfrak{m}_K(q_2)$ precisely when $\mathfrak{m}_L(q_1)$
contains the midpoint of $\mathfrak{m}_K(q_2)$, and this happens precisely when $d_f=-1/2$,
 since the distance between $\mathfrak{m}_L(q_2)$ and the $K$-vine of $q_2$ is $t_2-1/2$. 
Note that, in this case, one possible choice for the $K$-vine is
 the maximal path $\mathfrak{m}_L(q_1)$.

The case where $m_1$ is unramified and $m_2$ is ramified, is analogous, except that we replace the
path  $\mathfrak{m}_K(q_1)$ by a maximal path $\mathfrak{p}$ defined over an unramified extension $F/K$,
which can be converted into an $F$-vine by a suitable Moebius transformation.
In this case $d_f=-\frac{1}{2}$ is not possible, since $v_{1/2}$ in Fig. \ref{figure 2}\textbf{(B)}
is a point in a $K$-vine that is defined over a ramified quadratic extension of $K$, 
so it cannot belong to $\mathfrak{p}$, which contains a unique vertex defined over $K$.
The cases where $m_1$ is ramified but $m_2$ is not are also similar.

In the remaining case, either stem, $\mathfrak{m}_K(q_1)$ or $\mathfrak{m}_K(q_2)$, is an edge located in the 
$K$-vine of the corresponding idempotent, either $w_1$ or $w_2$, as in 
Fig. \ref{figure 2}\textbf{(B)}. Let 
$t_1, t_2 \in \mathbb{N}$ be the integers defined by
$\mathbb{D}\left(\frac{b_i}{a_i^2} \right) =(\pi^{-2t_i+1})$,  for $i\in\{1,2\}$,
 so that $\mathfrak{m}_L(q_i)$ is at a distance $t_i$ from either endpoint of $\mathfrak{m}_K(q_i)$. 
Then the $K$-stems are located as shown in \cite[Fig. 7]{a-cb1}, and we can reason precisely as in the ramified 
case of \cite[\S 6]{a-cb1}.

\section{Proof of Theorem \ref{t22} in the mixed case}\label{i-s}

In this section we assume that $m_1$ is a separable polynomial while $m_2$ is inseparable. Let $q_i$, $\alpha_i$ and $\lambda$ be as in \S4.1. Here  $\Delta=\lambda^2+a_1^2b_2$.

Again, we consider first the case where both polynomials split, so
$d_f=-\frac12\nu\left(\frac{\Delta}{a_1^2}\right)$.
As before, we can assume that $q_1,q_2$ are as given in Table \ref{tabla 2}, this time in row 6 or 5, 
according to whether the branches have a common end (as in Fig. 6(\textbf{A}-\textbf{B})) or not 
(as in Fig. 6(\textbf{C})).
\begin{figure}
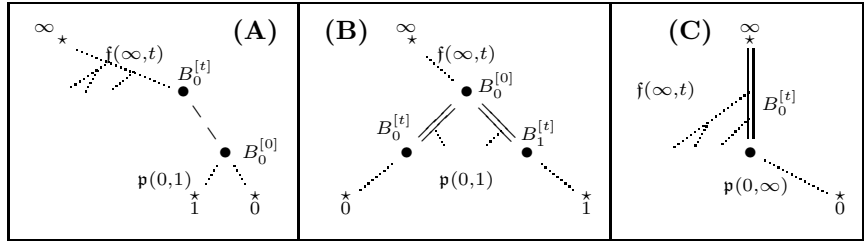

\[ 
\fbox{ \xygraph{
!{<0cm,0cm>;<.8cm,0cm>:<0cm,.8cm>::} 
 !{(0.5,3) }*+{{}^\infty}="e3"
!{(3,0) }*+{{}^1}="e4" !{(4,0) }*+{{}^0}="e5"
  !{(3.5,1) }*+{\bullet}="t3"   !{(2.8,2) }*+{\bullet}="t4"
 !{(2,2.6) }*+{{}^{\mathfrak{f}(\infty,t)}}="t7" !{(2.5,0.5) }*+{{}^{\mathfrak{p}(0,1)}}="t8"
!{(4.1,1) }*+{{}^{B_0^{[0]}}}="ti36"
!{(3,2.3) }*+{{}^{B_0^{[t]}}}="ti4"
!{(4,3) }*+{\textbf{(A)}}="ti99"
!{(0.8,2.8) }*+{{}^{\star}}="f3"
!{(3,0.2) }*+{{}^{\star}}="f4" !{(4,0.2) }*+{{}^{\star}}="f5"
!{(1.5,1.95) }*+{}="u3"  !{(0.8,2) }*+{}="u2"  !{(1.1,1.9) }*+{}="u1"  
!{(2.1,2.45) }*+{}="v3"  !{(1.7,2.6) }*+{}="v2"  !{(1.45,2.5) }*+{}="v1" 
"u3"-@{.}"v3" "u2"-@{.}"v2" "u1"-@{.}"v1"
   "t4"-@{.}"f3" "t3"-@{--}"t4" 
"t3"-@{.}"f4" "t3"-@{.}"f5"
 } }
\fbox{ \xygraph{
!{<0cm,0cm>;<.8cm,0cm>:<0cm,.8cm>::} 
 !{(0.5,0) }*+{{}^0}="e2" !{(1.5,3) }*+{{}^\infty}="e3"
 !{(4.5,0) }*+{{}^1}="e5"
 !{(2.5,2.6) }*+{{}^{\mathfrak{f}(\infty,t)}}="t7" !{(2.5,0.5) }*+{{}^{\mathfrak{p}(0,1)}}="t8"
  !{(2.5,2) }*+{\bullet}="t3"   !{(1.5,1) }*+{\bullet}="t4" !{(3.5,1) }*+{\bullet}="t5" 
!{(3.0,2.1) }*+{{}^{B_0^{[0]}}}="ti3"
!{(1.3,1.4) }*+{{}^{B_0^{[t]}}}="ti4"
!{(3.7,1.3) }*+{{}^{B_1^{[t]}}}="ti5"
!{(0.5,3) }*+{\textbf{(B)}}="ti99"
 !{(0.5,0.2) }*+{{}^{\star}}="f2" !{(1.6,2.8) }*+{{}^{\star}}="f3"
 !{(4.5,0.2) }*+{{}^{\star}}="f5"
 !{(2.7,1) }*+{}="u2"  !{(2.2,1) }*+{}="u1"  
  !{(3.2,1.5) }*+{}="v2"  !{(1.9,1.5) }*+{}="v1" 
 "u2"-@{.}"v2" "u1"-@{.}"v1"
   "f2"-@{.}"t4" "t3"-@{=}"t4" "f3"-@{.}"t3" "t3"-@{=}"t5"
 "t5"-@{.}"f5"
 } }
\fbox{ \xygraph{
!{<0cm,0cm>;<.8cm,0cm>:<0cm,.8cm>::} 
 !{(2.5,0) }*+{{}^0}="e2" !{(1,3) }*+{{}^\infty}="e3"
  !{(1,1) }*+{\bullet}="t3"   !{(0.4,1) }*+{}="u3"  !{(-0.4,1) }*+{}="u2"  !{(0,1) }*+{}="u1"  
!{(1.1,1.7) }*+{}="v3"  !{(1.1,2.1) }*+{}="v2"  !{(0.35,1.575) }*+{}="v1" 
 !{(-0.4,2.0) }*+{{}^{\mathfrak{f}(\infty,t)}}="t7" !{(1.1,0.4) }*+{{}^{\mathfrak{p}(0,\infty)}}="t8"
!{(1.5,1.8) }*+{{}^{B_0^{[t]}}}="ti3"
!{(0.0,3) }*+{\textbf{(C)}}="ti99"
 !{(2.5,0.2) }*+{{}^{\star}}="f2" !{(1,2.8) }*+{{}^{\star}}="f3"
  "t3"-@{.}"f2" "e3"-@{=}"t3" "u3"-@{.}"v3" "u2"-@{.}"v2" "u1"-@{.}"v1"
 } }
\]
\caption{The different possible configurations of the path and the infinite foliage in the split mixed 
case (\S7).}\label{figure 5}
\end{figure}
Note that $\Delta=0$, and hence $d_f=\infty$, only in row $6$, which is precisely the case where
the intersection, depicted as a double line in Fig. \ref{figure 5}\textbf{(C)}, is a ray. 
In any other case we have  $d_f=-\frac12\nu\left(\frac{\Delta}{a_1^2}\right)
=-\nu(u\pi^t)=-t$.
If $\mathfrak{m}_K(q_1)\cap \mathfrak{m}_K(q_2)=\emptyset$, then $-t>0$ is the distance between the 
stems, as it is the length of the dashed line in Fig. \ref{figure 5}\textbf{(A)}. 
In the remaining case, the length of the intersection, denoted by a double line in Fig. \ref{figure 5}\textbf{(B)},
is $2t = -2d_f$.

In all that follows, $L$ is an algebraic extension of $K$ where both $m_1$ and $m_2$ split. 
First we consider the case where $m_1$ is unramified and $m_2$ splits in $K$. 
In this case, $\mathfrak{m}_K(q_1)$ consists only on the highest vertex in $\mathfrak{m}_L(q_1)$, 
and this is the unique vertex in $\mathfrak{m}_L(q_1) \cap \mathfrak{t}(K)$, as in Fig. \ref{figure 2}\textbf{(A)} 
(c.f. \S\ref{s4}). Assume first that $\mathfrak{m}_L(q_1)$ and $\mathfrak{m}_L(q_2)$ fail to intersect.
The minimal path $\mathfrak{p}$ from $\mathfrak{m}_K(q_1)$ to $\mathfrak{m}_L(q_2)$ is defined over $K$
 by Lemma 1.3. We claim that $\mathfrak{p}$  is also the minimal path from $\mathfrak{m}_L(q_1)$ to
$\mathfrak{m}_L(q_2)$. In fact, as $L/K$ is unramified, every vertex in $\mathfrak{p}$ is defined over $K$
 by Lemma 1.1. Since $\mathfrak{m}_L(q_1)\cap\mathfrak{t}(K)$ has a unique vertex, the claim follows,
whence $d(\mathfrak{m}_K(q_1), \mathfrak{m}_K(q_2))=d(\mathfrak{m}_L(q_1), \mathfrak{m}_L(q_2))=
d_f$. When $\mathfrak{m}_L(q_1)$ and $\mathfrak{m}_L(q_2)$ do intersect, a similar argument shows that
$\mathfrak{m}_L(q_1) \cap\mathfrak{m}_L(q_2)$ contains at most one vertex, so that $d_f=0$. 

Next assume $m_1$ ramifies and $m_2$ factors. Here $\mathfrak{m}_K(q_1)$ consists in an edge whose 
midpoint $v_{1/2}$  is at distance $t_1-\frac{1}{2}$ from $\mathfrak{m}_L(q_1)$, where 
$\mathbb{D}\left(\frac{b_1}{a_1^2} \right)=(\pi^{-2t_1+1})$, again as in Fig. \ref{figure 2}\textbf{(B)}.
 The end of the infinite foliage is defined over $K$,
whence it can be assumed to be $\infty$ by applying a suitable Moebius transformation defined over $K$. 
The ends of the maximal path, however, are conjugates, and defined over a ramified separable extension.
We can replace the infinite foliage $\mathfrak{m}_L(q_2)$ with a suitable fake branch 
$\mathfrak{m}_L(\pi_K^{-n}q_2)$, as in \S5, that lies at a distance 
$- \frac12\nu\left( \frac{\Delta}{a_1^2}\right)+n$ from $\mathfrak{m}_L(q_1)$,
for $n$ big enough. Let $B$ be the leaf of this fake branch that is closest to $\mathfrak{m}_K(q_1)$,
as in Fig. \ref{figure 6}\textbf{(A)}. 
It is clear from the picture that the distance from  $B$ to $\mathfrak{m}_K(q_1)$ is
$- \frac12\nu\left( \frac{\Delta}{a_1^2}\right)+n-t_1$. Now the following facts are apparent:
\begin{itemize} 
\item If $\mathfrak{m}_K(q_1)$ does not intersect $\mathfrak{m}_K(q_2)$, so we can set $n=0$,
 we have  $$d(\mathfrak{m}_K(q_1),\mathfrak{m}_K(q_2))= 
- \frac{1}{2}\nu\left( \frac{\Delta}{a_1^2}\right)-t_1=d_f.$$
\item If $d_f=0$, the intersection is a vertex.
\item $d_f$ cannot be $-\frac{1}{2}$, as $B$ is defined over $K$.
\item If $d_f\leq -1$, then $\mathfrak{m}_K(q_1) \subset \mathfrak{m}_K(q_2)$. 
\end{itemize} 
The result follows in this case.

Now, we assume that $m_1$ splits and $m_2$ is ramified. So that $L=K(\alpha_2)$
is an inseparable (ramified) extension. Then, again, the distance from the edge 
$\mathfrak{m}_K(q_2)$ to a leaf of $\mathfrak{m}_L(q_2)$ is $t_2$, where 
$\delta\left(b_2 \right) =(\pi^{2t_2+1})$. In this case the (ghost) midpoint 
$v_{1/2}'$ of $\mathfrak{m}_K(q_2)$ is a leaf of a fake branch 
$\mathfrak{f}(q_2)=\mathfrak{m}_L(q'_2)$, for $q'_2=\pi_L^{-1-2t_2}(q_2-\alpha_2)$, 
as in \S5. The distance from $\mathfrak{f}(q_2)$ to  $\mathfrak{m}_L(q_1)$ is, by applying
the split case to $q_1$ and $q_2'$, the following:
$$ -\frac{1}{2}\nu\left( \frac{\Delta}{a_1^2}\right)+t_2+\frac12= d_f+\frac12.$$
Furthermore, if $\mathfrak{p}$ is the maximal path in $L$ identified with $\mathfrak{m}_K(q_1)$,
then  the path $\mathfrak{p}_1$  from $v_{1/2}'$ to $\mathfrak{p}$ cannot pass through any other 
vertex of the fake branch $\mathfrak{f}(q_2)$, as the latter contains no point defined over $K$,
while the vertex after $v_{1/2}'$ in $\mathfrak{p}_1$ must be defined over $K$ by Lemma 1.1 and 
Lemma 1.3, as $v_{1/2}'$ has neighbors defined over $K$ at distance $1/2$.
Therefore, $\mathfrak{p}_1$ is either trivial, as in Fig.  \ref{figure 6}\textbf{(B)},
 or it contains an endpoint of $\mathfrak{m}_K(q_2)$. The following facts are now apparent:
\begin{itemize}
 \item If $\mathfrak{m}_L(q_1)=\mathfrak{m}_K(q_1)$ does not contain the edge 
$\mathfrak{m}_K(q_2)$, then the distance  between these stems is $d_f\geq0$.
\item  $\mathfrak{m}_L(q_1)$ contains $\mathfrak{m}_K(q_2)$ precisely when $v_{1/2}'$ is in 
$\mathfrak{m}_L(q_1)$, as in Fig. \ref{figure 6}\textbf{(B)}, and in this case $d_f=-1/2$.
\end{itemize}
Note that $d_f=-1/2$ is the minimum possible value in this case.
The case where $m_1$ is unramified and $m_2$ is ramified, is analogous, except that in this case $d_f=-\frac{1}{2}$
 is not possible, as both, the midpoint $v_{1/2}'$, in an edge defined over $K$,  and the path $\mathfrak{m}_L(q_1)$,
with a unique vertex defined over $K$, are defined over different quadratic extensions.

\begin{figure}
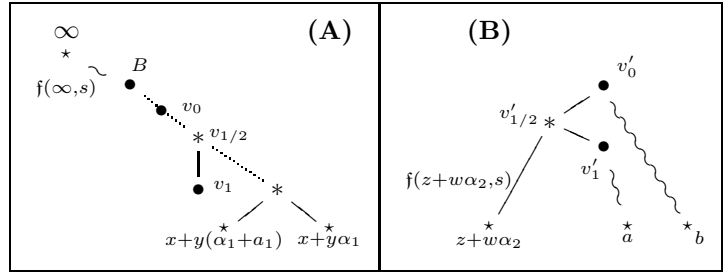

\[ 
\fbox{ \xygraph{
!{<0cm,0cm>;<.7cm,0cm>:<0cm,.7cm>::} 
!{(0.0,4) }*+{\infty}="e3"
!{(1.2,3) }*+{\bullet}="x3"
 !{(0,2.9) }*+{{}^{\mathfrak{f}(\infty,s)}}="f5" !{(1.4,3.3) }*+{{}^{B}}="f7" 
 !{(5,0) }*+{{}^{x+y\alpha_1}}="e2" 
!{(3,0) }*+{{}^{x+y(\alpha_1+a_1)}}="e4" 
  !{(2.5,2) }*+{*}="t3"    !{(1.8,2.5)}*+{\bullet}="t6"    !{(2.5,1) }*+{\bullet}="t7" 
  !{(4,1) }*+{*}="t4" 
!{(3.1,2) }*+{{}^{ v_{1/2}}}="ti3" !{(2.4,2.5) }*+{{}^{v_0}}="ti3" !{(3,1.) }*+{{}^{v_1}}="ti3"
!{(5,4) }*+{\textbf{(A)}}="ti99"
!{(0,3.5) }*+{{}^{\star}}="f3"  !{(5,0.2) }*+{{}^{\star}}="f2" 
!{(3,0.2) }*+{{}^{\star}}="f4" 
   "f2"-"t4" "t3"-@{.}"t4" "t3"-@{.}"x3" "f3"-@{~}"x3"  
"t4"-"f4" "t3"-@{-}"t7" "t3"-@{-}"t6"}}
\fbox{ \xygraph{
!{<0cm,0cm>;<.8cm,0cm>:<0cm,.8cm>::} 
 !{(0.5,-0.02) }*+{{}^{z+w\alpha_2}}="e2"  !{(2.8,0) }*+{{}^{a}}="e5"
!{(1.5,2) }*+{*}="t3" !{(1.0,2.1) }*+{{}^{v_{1/2}'}}="t10"   !{(2.4,2.6) }*+{\bullet}="t4"  !{(2.4,1.6) }*+{\bullet}="t5"  !{(0,1) }*+{{}^{\mathfrak{f}(z+w\alpha_2,s)}}="e6" 
!{(4,0) }*+{{}^{b}}="t6"  !{(2.8,2.9) }*+{{}^{v'_0}}="t8"  !{(2.2,1.2) }*+{{}^{v'_1}}="t9" 
!{(2.6,2) }*+{{}^{}}="ti3"
!{(0.5,3.5) }*+{\textbf{(B)}}="ti99"
 !{(0.5,0.2) }*+{{}^{\star}}="f2"  !{(2.8,0.2) }*+{{}^{\star}}="f5" !{(3.8,0.2) }*+{{}^{\star}}="f6"  
   "f2"-"t3" "t3"-@{-}"t5" "f5"-@{~}"t5" "t3"-@{-}"t4" "f6"-@{~}"t4" 
 } }
\]
\caption{In \textbf{(A)}, $B$ is a leaf of $\mathfrak{m}_K(\pi_K^{-n}q_2)$ and 
$\lbrace v_0,v_1\rbrace$ are the vertices $\mathfrak{m}_K(q_1)$. In \textbf{(B)} 
the vertices of $\mathfrak{m}_K(q_2)$ are $v'_0$ and $v'_1$.}\label{figure 6}
\end{figure}

In the remaining case, either $K$-stem, $\mathfrak{m}_K(q_1)$ or $\mathfrak{m}_K(q_2)$, consist in an edge 
located in the $K$-vine minimizing the fake distance to the corresponding $L$-stem, 
as in Fig. \ref{figure 2}\textbf{(B)} and Fig. \ref{figure 1}\textbf{(B)}, respectively. 
Let $t_1, t_2 \in \mathbb{N}$ the integers satisfying 
$\mathbb{D}\left(\frac{b_1}{a_1^2} \right) =(\pi^{-2t_1+1})$ and 
$\delta\left(b_2 \right) =(\pi^{2t_2+1})$. We define the fake branch of $\mathfrak{s}_K(q_2)$ by 
$\mathfrak{f}(q_2) = \mathfrak{s}_L(q''_2)$, where, $q_2''=\pi^{-t_2-1} (q_2-\alpha_2)$, as before. 
Recall from \S5 that a leaf of this fake branch is located as one of the white circles in Fig. 4. 
By applying the computation for the split case, this time to $q_1$ and  $q_2''$,
we conclude that the distance from $\mathfrak{f}(q_2)$ to $\mathfrak{m}_L(q_1)$ is
$-\frac12\nu\left(\frac{\Delta}{a_1^2}\right)+t_2+1=d_f+t_1+1$,
unless this amount is negative, and in the latter case $\mathfrak{m}_L(q_1)$ intersects $\mathfrak{f}(q_2)$
non-trivially. The upcoming reasoning is analog  to the inseparable case, as depicted in Fig. 4, except that
in this case the $L$-stem of the separable quaternion $q_1$  is at distance $t_1-1/2$ behind the corresponding
$K$-stem. If the two $K$ stems  $\mathfrak{m}_K(q_1)$ and $\mathfrak{m}_K(q_2)$ are different, even if they 
intersect at one point,  the path from $\mathfrak{f}(q_2)$ to $\mathfrak{m}_L(q_1)$ 
must pass though one endpoint of each $K$-stem. The branch $\mathfrak{f}(q_2)$ is at distance 
$1/2$ from the midpoint of  $\mathfrak{m}_K(q_2)$, so the distance from this white point to the midpoint of 
$\mathfrak{m}_K(q_1)$ is larger that the distance between the $K$-stems by $3/2$.
The distance between the $K$-stems is therefore given by
$(d_f+t_1+1)-(t_1-1/2)-3/2=d_f$, whence the the result follows. 
If the $K$-stems coincide and $d_f+t_1+1\geq0$, Fig. 4(\textbf{C}) shows that the path
from $\mathfrak{f}(q_2)$ to $\mathfrak{m}_L(q_1)$ cannot be longer than $1/2+(t_1-1/2)=t_1$,
so $d_f<0$. The same conclusion holds if   $d_f+t_1+1<0$. This finishes the proof.

\section{On the representations of some algebras}\label{s6}

In this section we prove Theorem \ref{t24} using tools from representation theory.
Let $\lambda\in K$, let $m_i(X) = X^2+a_iX+b_i\in\mathcal{O}[X]$, for $i\in\{1,2\}$, and define $\Delta$
as in Theorem 2.2.  In all of \S8, we consider the $K$-algebra $\mathcal{A}=\mathcal{A}(\lambda,m_1,m_2)$, 
defined in terms of generators and relations as follows:
\begin{equation}\label{eq algebra}
\mathcal{A}=K \Big[\q_1,\q_2\Big|m_1(\q_1)= m_2(\q_2)=0,\ \q_1(\q_2+a_2)+\q_2(\q_1+a_1)=\lambda\Big].
\end{equation}

\begin{lemma}\label{lema 6.1}
The algebra $\mathcal{A}$ defined in (\ref{eq algebra}) is a $4$-dimensional $K$-algebra.
 It is a quaternion algebra if and only if $\Delta\neq 0$. 
\end{lemma}

\begin{proof}
First, we assume that $a_1\neq 0$ and $\Delta\neq 0$. Let 
$\q_2'=\lambda+a_2 \q_1+ a_1\q_2$ and $\q_1'=\frac{\q_1}{a_1}$. It is apparent that
$$\mathcal{A}= K\left[ \q_1',\q_2'\Big| \q_1'^2+\q_1'+\frac{b_1}{a_1^2}=\q_2'^2+\Delta=
 \q_2'\q_1'+(\q_1'+1)\q_2'=0\right] $$
is a cyclic algebra, and therefore a quaternion algebra 
(c.f. \cite[Ch. 1, \S 1, Ex. 1.6]{vigneras}). Now, we assume that $a_1=a_2=0$ 
and $\Delta=\lambda^2\neq 0$. If $b_2=b_1=0$, we can replace $\q_2$ by 
 $\widehat{\q_2}=\q_2+1$, and with the new generators $\mathcal{A}=\mathcal{A}(\lambda, X^2, X^2+1)$,
so we assume $b_2 \neq 0$. In the latter case, $\q_1''=\frac{\q_1\q_2}{\lambda}$ gives the presentation 
$$\mathcal{A}= K\left[ \q_1'',\q_2\Big|\q_1''^2+\q_1''+\frac{b_1b_2}{\lambda^2}=  \q_2^2+b_2= \q_1''\q_2+\q_2(\q_1''+1)=0\right],$$ 
which is again a cyclic algebra. 

In the rest of the proof we assume $\Delta=0$. First consider the case $a_1 \neq 0$. 
Define $\q_1'$ and $\q_2'$ as above, and let $\alpha'_1\in\overline{K}$ be a root of the irreducible polynomial 
$m'_1(X)=X^2+X+\frac{b_1}{a_1^2}$ of $\q'_1$. Set $L=K(\alpha'_1)$, so that there exists  a representation
$\phi:\mathcal{A}_L=\mathcal{A}\otimes_K L \rightarrow \mathbb{M}_4(L)$ 
defined by
$$ \phi(\q_1'\otimes 1)= \textnormal{\scriptsize{$ \left( \begin{array}{cccc}
\alpha'_1+1 & 0  &0 & 0\\
0 & \alpha'_1 & 0& 0 \\
0 & 0 & \alpha'_1 &0 \\
0 & 0 & & \alpha'_1+1 \end{array} \right) $\normalsize}} , \quad \phi(\q_2'\otimes 1)=
\textnormal{\scriptsize{$ \left( \begin{array}{cccc}
0 & 0  &0 & 0\\
1 & 0 & 0& 0 \\
0 & 0 & 0 &0 \\
0 & 0 &1 & 0 \end{array} \right) $\normalsize}},
$$
whose image is $4$-dimensional, and has a non-trivial radical $R(\mathcal{A})$. 
This proves that $\mathcal{A}$ is not a quaternion algebra and $\dim_K\mathcal{A}\geq 4$.
The converse inequality follows from the fact that $\{1,\q_1,\q_2,\q_1\q_2\}$ spans $\mathcal{A}$
as a vector space. The same argument holds if $a_2\neq 0$. 
On the other hand, if $a_1=a_2=0$, then $0=\Delta=\lambda^2$, and therefore $\lambda=0$.
Replacing these values in (\ref{eq algebra}), we obtain that $\mathcal{A}$ is a $4$-dimensional commutative algebra.
\end{proof}

\begin{proposition}\label{t23} Assume $\Delta \neq 0$, so $\mathcal{A}$ is a quaternion algebra. Then
 $\mathcal{A} \cong \mathbb{M}_2(K)$ if and only if at least one of the following conditions holds:
\begin{itemize}
\item[i.-] $m_1(X)$ or $m_2(X)$ has a zero in $K$,
\item[ii.-] $\lambda=C(x,y,z,w)$ for some pairs $(x,y),(z,w) \in K \times K^{*}$, as in \eqref{eq2}.
\end{itemize}
\end{proposition}

\begin{proof}
It follows from the Lemma \ref{lema 6.1} that $\mathcal{A} \cong \mathbb{M}_2(K)$ or $\mathcal{A}$ is a 
quaternion division 
algebra (c.f. \cite[\S 52 E, Theo. 52.9]{Om}). Suppose that $m_1$ has a root $\alpha_1 \in K$. 
If $a_1 \neq 0$, then $\frac{\q_1+\alpha_1}{a_1}$ is a nontrivial idempotent. If $a_1=0$, then $\q_1+\alpha_1$
 is a nontrivial nilpotent. Either implies $\mathcal{A} \cong \mathbb{M}_2(K)$. Now, we assume the existence
of  two pairs $(x,y),(z,w) \in K \times K^{*}$ satisfying $\lambda=C(x,y,z,w)$. Consider the pair
of non-zero conjugates $\q=(x-z)+y\q_1+w\q_2$ and $\overline{\q}=(x-z)+y(\q_1+a_1)+w(\q_2+a_2)$ 
in $\mathcal{A}$. 
A simple computation shows that $ \q\overline{\q}=\overline{\q}\q= yw(C(x,y,z,w)+\lambda)=0$.
We conclude that $\mathcal{A} \cong \mathbb{M}_2(K)$ again. 
Reciprocally, assume the existence of an isomorphism $\phi: \mathcal{A} \rightarrow \mathbb{M}_2(K)$, 
and assume that neither $m_1(X)$ nor $ m_2(X)$  have a zero in $K$. Let $\q_1,\q_2 \in \mathcal{A}$ be 
generators of $\mathcal{A}$ as given in \eqref{eq algebra}, and let $q_1,q_2$ be the corresponding 
matrices. Since $ \mathbb{M}_2(K)$ contains a two dimensional space
of norm $0$ elements, for example, the matrices with a vanishing first row, we can always find there
an element of the form  $q=(x-z)+yq_1+wq_2$. Neither $y$ nor $w$ can vanish, as either element in
$\{q_1,q_2\}$ generates a field. The result follows.
\end{proof}

\begin{example}
We claim that, if  $m_2=x^2+\pi$ and the splitting field
of $m_1$ is an unramified quadratic extension of $K$, then $\mathcal{A}=\mathcal{A}(0,m_1,m_2)$
 is a division algebra. Note that $\lambda=a_2=0$ implies $\Delta= a_1^2 \pi \neq 0$. In fact, if 
$\mathcal{A}$ were a matrix algebra, then there would exist pairs $(x,y),(z,w) \in K \times K^{*}$ satisfying $C(x,y,z,w)=0$, or equivalently  
$$(yw)C(x,y,z,w)= x^2+a_1 xy+b_1 y^2+z^2+ \pi w^2+a_1 zy=0.$$
Rearranging, we obtain $\Big|y^2m_1\Big(\frac{x+z}y\Big)\Big|=|\pi w^2|$, which is impossible
by the properties of the Artin-Schreier defect, or equivalently, since
$m_1$ has no roots in the residue field. 
We conclude that $\mathcal{A}$ is a division algebra. This is a characteristic-2 analog of
\cite[\S 63 B, Theo. 63.11 B]{Om}.
\end{example}

\begin{proof}[Proof of Theorem 2.3]
Assume first $\Delta \neq 0$, and assume the existence of elements $q_1,q_2 \in \mathbb{M}_2(K)$ satisfying 
\eqref{133}. Let $\varphi: \mathcal{A}=\mathcal{A}(\lambda,m_1,m_2) \rightarrow \mathbb{M}_2(K)$ be the 
representation defined by $\q_i\mapsto q_i$. 
 As $\mathcal{A}$ is a simple algebra and $\dim_K \mathcal{A}= \dim_K \mathbb{M}_2(K)$, 
$\varphi$ must be an isomorphism. Reciprocally, if $\varphi: \mathcal{A} \rightarrow \mathbb{M}_2(K)$ is an 
isomorphism, then the images $q_1=\varphi(\q_1)$ and $q_2= \varphi(\q_2)$ satisfy \eqref{133}. 
The results follows in this case from Prop. \ref{t23}.

In all that follows we assume $\Delta=0$. Now suppose $a_1 \neq 0$ and assume the existence of 
$q_1,q_2 \in \mathbb{M}_2(K)$ satisfying \eqref{133}. Define $q_2''=a_2q_1+a_1q_2=q_2'+\lambda$, 
where $q_2'$ corresponds to the element $\q'_2$ in  the proof of Lemma \ref{lema 6.1},  
so that $q_2''^2=\lambda^2$ and
\begin{equation}\label{eq5}
q_1q_2''+q_2''(q_1+a_1)=a_1\lambda. 
\end{equation}
Changing the base if needed, we can assume $ q_2''= \textnormal{\scriptsize{$ \left( \begin{array}{cc}
\lambda & 0  \\1 & \lambda  \end{array} \right) $\normalsize}}$. Set $q_1= 
\textnormal{\scriptsize{$\left( \begin{array}{cc} u & y  \\  z & w\end{array} \right) $\normalsize}}$,
so that identity (\ref{eq5}) gives $y=0$ and $w=u+a_1$. Then, the condition $m_1(q_1)=0$ implies $m_1(u)=0$. Conversely, if $m_1$ has a root $\alpha \in K$, then the quaternions 
$q_1=\textnormal{\scriptsize{$ \left( \begin{array}{cc}
a_1+\alpha & 0  \\
0 & \alpha\end{array} \right)$\normalsize}}$
and  $q_2=  \textnormal{\scriptsize{$\frac{1}{a_1}\left(\begin{array}{cc}
\lambda+ a_1 a_2 + a_2 \alpha & 0  \\
1 & \lambda+ a_2 \alpha  \end{array} \right) $\normalsize}}$ satisfy \eqref{133}. 
The same holds if $a_2\neq 0$.

Finally, assume $a_1=a_2 =0$, so that $\lambda=0$,
 and set $L=K(\sqrt{b_1},\sqrt{b_2})$, as in the proof of  Lemma \ref{lema 6.1}.
Is not hard to see, as  $\sqrt{b_1}$ and $\sqrt{b_2}$ satisfy the relations defining $\mathcal{A}$,
that $L$ is isomorphic to a quotient of $\mathcal{A}$.
If $[L:K]=4$, so that $\mathcal{A} \cong L$, the algebra $\mathcal{A}$ 
cannot have a two dimensional representation. 
If $[L:K]\leq2$, such representations do exists. Note however that any pair of matrices $(q_1,q_2)$
satisfying \eqref{133} must be contained in a commutative subalgebra, and in $\matrici_2(K)$ such
algebras are always two dimensional.
The result follows if we prove that $[L:K]\leq2$ in our case. In fact, any Laurent series can be written in the
form $$f(\pi)=\sum_{m\in\mathbb{Z}}a_{m}\pi^{m}
=\sum_{n\in\mathbb{Z}}b_{2n}^2\pi^{2n}+\pi\sum_{n\in\mathbb{Z}}b_{2n+1}^2\pi^{2n},$$
where $b_m^2=a_m$, so $K$ has a unique inseparable quadratic extension. Note that this is the
only point in the proof where we use the fact that the residue field is perfect.
\end{proof}

\section{acknowledgments}

The first author was supported by Fondecyt, Project No 1200874. The second authors was supported by CONICYT,
Doctoral scholarship No 21180544.

$$ $$

Luis Arenas-Carmona

Departamento de Matem\'aticas, Facultad de Ciencias

Universidad de Chile, Casilla 653, Santiago, Chile

learenas@u.uchile.cl

$$ $$

Claudio Bravo

Departamento de Matem\'aticas, Facultad de Ciencias

Universidad de Chile, Casilla 653, Santiago, Chile

claudio.bravo.c@ug.uchile.cl

\end{document}